\documentclass[english]{article}
\usepackage[T1]{fontenc}
\usepackage[latin9]{inputenc}
\usepackage{amsthm}
\usepackage{amsmath}
\usepackage{amssymb}
\usepackage{esint}

\makeatletter
\theoremstyle{plain}
\newtheorem{thm}{\protect\theoremname}[section]
  \theoremstyle{plain}
  \newtheorem{prop}[thm]{\protect\propositionname}
  \theoremstyle{plain}
  \newtheorem{cor}[thm]{\protect\corollaryname}
  \theoremstyle{definition}
  \newtheorem{defn}[thm]{\protect\definitionname}
  \theoremstyle{remark}
  \newtheorem{rem}[thm]{\protect\remarkname}
  \theoremstyle{plain}
  \newtheorem{lem}[thm]{\protect\lemmaname}
  \theoremstyle{remark}
  \newtheorem{claim}[thm]{\protect\claimname}
  \theoremstyle{remark}
  \newtheorem*{claim*}{\protect\claimname}

\makeatother

\usepackage{babel}
  \providecommand{\claimname}{Claim}
  \providecommand{\corollaryname}{Corollary}
  \providecommand{\definitionname}{Definition}
  \providecommand{\lemmaname}{Lemma}
  \providecommand{\propositionname}{Proposition}
  \providecommand{\remarkname}{Remark}
\providecommand{\theoremname}{Theorem}

\begin{document}
\global\long\def\R{\mathbb{R}}

\global\long\def\C{\mathbb{C}}

\global\long\def\Z{\mathbb{Z}}

\global\long\def\P{\mathbb{P}}

\global\long\def\F{\mathbb{F}}

\global\long\def\Dens{\mbox{{Dens}}}

\global\long\def\M{\mathcal{M}}

\global\long\def\sign{\mbox{sign}}

\global\long\def\Stab{\mbox{Stab}}

\global\long\def\Rad{\mbox{\ensuremath{\mathcal{R}}}}

\global\long\def\eps{\varepsilon}
 \global\long\def\alp{\alpha}
 \global\long\def\ome{\omega}
 \global\long\def\Ome{\Omega}
 \global\long\def\lam{\lambda}
 \global\long\def\Lam{\Lambda}
 \global\long\def\gam{\gamma}
 \global\long\def\to{\rightarrow}
 \global\long\def\qed{ Q.E.D. }
 \global\long\def\pt{\partial}
 \global\long\def\gr{{}\!^{\textbf{R}} Gr}
 \global\long\def\grc{{}\!^{\textbf{C}} Gr}
 \global\long\def\si{{}\!^{S}\int}
 \global\long\def\sio{{}\!^{S_{1}}\int}
 \global\long\def\sit{{}\!^{S_{2}}\int}
 \global\long\def\oid#1#2#3{\mbox{\ensuremath{{\cal #1}_{#2}^{#3}}}}
 \global\long\def\idn#1#2#3{\mbox{\ensuremath{{\bf #1}_{#2}^{#3}}}}
 \global\long\def\idrel#1#2#3{\mbox{\ensuremath{{\cal #1}\stackrel{1}{#2}{\cal #3}}}}
 \global\long\def\RR{\mathbb{R}}
 \global\long\def\CC{\mathbb{C}}
 \global\long\def\QQ{\mathbb{Q}}
 \global\long\def\NN{\mathbb{N}}
 \global\long\def\ZZ{\mathbb{Z}}
 \global\long\def\HH{\mathbb{H}}
 \global\long\def\DD{\mathbb{D}}
 \global\long\def\PP{\mathbb{P}}
 \global\long\def\fa{ f_{a} }
 \global\long\def\nuf{ \nu_{\fa} (\eps) }
 \global\long\def\bl{ balls}
 \global\long\def\an{ \oid A{}{} ( \CC^{n} ) }
 \global\long\def\as{ \oid A{}{} ( \CC^{s} ) }
\global\long\def\One{{1\hskip-2.5pt{\rm l}}}
 \swapnumbers 

\global\long\def\tphi{\tilde{\phi}}
 \global\long\def\tpsi{\tilde{\psi}}
 \global\long\def\tf{\tilde{f}}
 \global\long\def\tg{\tilde{g}}

\global\long\def\ca{{\cal A}}
 \global\long\def\cb{{\cal B}}
 \global\long\def\cc{{\cal C}}
 \global\long\def\cd{{\cal D}}
 \global\long\def\ce{{\cal E}}
 \global\long\def\cf{{\cal F}}
 \global\long\def\cg{{\cal G}}
 \global\long\def\ch{{\cal H}}
 \global\long\def\ci{{\cal I}}
 \global\long\def\cj{{\cal J}}
 \global\long\def\ck{{\cal K}}
 \global\long\def\cl{{\cal L}}
 \global\long\def\cm{{\cal M}}
 \global\long\def\cn{{\cal N}}
 \global\long\def\co{{\cal O}}
 \global\long\def\cp{{\cal P}}
 \global\long\def\cq{{\cal Q}}
 \global\long\def\car{{\cal R}}
 \global\long\def\cs{{\cal S}}
 \global\long\def\ct{{\cal T}}
 \global\long\def\cu{{\cal U}}
 \global\long\def\cv{{\cal V}}
 \global\long\def\cw{{\cal W}}
 \global\long\def\cx{{\cal X}}
 \global\long\def\cy{{\cal Y}}
 \global\long\def\cz{{\cal Z}}

\textheight=9in \topmargin=0pt\headheight=0pt\headsep=0pt \textwidth=6.5in
\oddsidemargin=0pt 
\global\long\def\inj{\hookrightarrow}
 \global\long\def\surj{\twoheadrightarrow}
 \global\long\def\vi{V^{\infty}}
 \global\long\def\vmi{V^{-\infty}}
 \global\long\def\vic{V_{c}^{\infty}}
 \global\long\def\vmic{V_{c}^{-\infty}}
 \global\long\def\vix{V^{\infty}(X)}
 \global\long\def\vmix{V^{-\infty}(X)}
 \global\long\def\vicx{V_{c}^{\infty}(X)}
 \global\long\def\vmicx{V_{c}^{-\infty}(X)}
 \global\long\def\supp{\operatorname{supp} \,}
 \global\long\def\htimes{\hat{\otimes}}
\textheight=9in \topmargin=0pt\headheight=0pt\headsep=0pt \textwidth=6.5in
\oddsidemargin=0pt 

\title{Convex valuations invariant under the Lorentz group}
\date{}

\author{Semyon Alesker and Dmitry Faifman%
\thanks{The authors were partially supported by ISF grants 701/08 and 1447/12%
}}
\maketitle
\begin{abstract}
We give an explicit classification of translation-invariant, Lorentz-invariant
continuous valuations on convex sets. We also classify the Lorentz-invariant
even generalized valuations.
\end{abstract}

\section{Introduction}

The main result of this paper is to give a complete classification
of translation invariant continuous valuations on convex sets in $\RR^{n}$
invariant under the connected component of the Lorentz group.

Let $\ck^{n}$ denote the family of convex compact subsets of $\RR^{n}$.
A (convex) valuation is a functional $\phi\colon\ck^{n}\to\CC$ which
satisfies the following additivity property 
\[
\phi(A\cup B)=\phi(A)+\phi(B)-\phi(A\cap B)
\]
whenever $A,B,A\cup B\in\ck^{n}$. A valuation is called continuous
if it is continuous with respect to the Hausdorff metric on $\ck^{n}$.

\hfill{}

Classification results are playing an important role in the valuations
theory and its applications to integral geometry since the fundamental
work of Hadwiger in the 1940's and 1950's. Probably the most famous
result in the area is Hadwiger's characterization \cite{hadwiger-book}
of continuous valuations on convex subsets of a Euclidean space invariant
under all isometries, i.e. translations and all orthogonal transformations,
as linear combinations of intrinsic volumes (see \cite{schneider-book}
for this notion); the subgroup of orientation preserving isometries
leads to the same list of invariant valuations. In recent years many
new classification results have been obtained for various classes
of valuations. Thus Klain \cite{klain} and Schneider \cite{schneider-simple}
have classified continuous translation invariant valuations which
are simple, i.e. vanish on convex sets of positive codimension. In
\cite{alesker-adv-2000} the first author have proven the following
general results: let $G$ be a compact subgroup of the linear group.
The subspace of $G$-invariant translation invariant continuous valuations
on convex sets is finite dimensional if and only if $G$ acts transitively
on the unit sphere; thus for such a group $G$ one may hope to get
certain finite classification. The problem to obtain such a classification
is under investigation of a few people in recent year. Notice that
the cases $G=O(n),SO(n)$ correspond to the Hadwiger theorem. The
next interesting case $G=U(n)$ was classified explicitly in geometric
terms by the first author \cite{Alesker-jdg} where also first applications
to Hermitian integral geometry were obtained. More thorough and complete
further study of $U(n)$-invariant valuations and Hermitian integral
geometry was done by Bernig and Fu \cite{bernig-fu-annals} and Fu
\cite{fu-jdg-algebra}. Several other cases of compact groups acting
transitively on the sphere were considered by Bernig \cite{bernig-sun},
\cite{bernig-exception}, \cite{bernig-quaternion}.

At the same time other classes of valuations were studied under weaker
assumptions on continuity but stronger assumptions on the symmetry
group, which usually was either $GL_{n}(\RR)$ or $SL_{n}(\RR)$.
Thus Ludwig and Reitzner \cite{ludwig-reitzner-99} have characterized
the affine surface area as the only (up to volume, Euler characteristic,
and non-negative multiplicative factor) upper semi-continuous valuation
invariant under affine volume preserving transformations. Other results
on $SL_{n}(\RR)$-invariant valuations were obtained again by Ludwig
and Reitzner \cite{ludwig-reitzner-annals2010}. Quite a few of classification
results in a different but related direction of convex body valued
valuations were obtained in \cite{ludwig-adv-math-2002}, \cite{ludwig-amer-j-math-2006},
\cite{schneider-schuster-imnr-2006}, \cite{schuster-wannerer}; see
also references therein.

\hfill{}

Let us now discuss in greater detail the main results of the present
paper. Let us fix on $\RR^{n}$ the Minkowski metric, i.e. sign indefinite
quadratic form $Q$ of signature $(n-1,1)$. In coordinates it is
given by $Q(x)=\sum_{i=1}^{n-1}x_{i}^{2}-x_{n}^{2}$. Let $O(n-1,1)$
denote the group of all linear transformations of $\RR^{n}$ preserving
$Q$. It is well known that $O(n-1,1)$ has four connected components.
Let us denote by $SO^{+}(n-1,1)$ the connected component of the identity.
Throughout the article, we refer to $SO^{+}(n-1,1)$ as the Lorentz
group.

Let us denote by $Val(\RR^{n})$ the space of all translation invariant
continuous valuations on $\RR^{n}$. For an integer $k$ let us denote
by $Val_{k}(\RR^{n})$ the subspace of $k$-homogeneous valuations
(a valuation $\phi$ is called $k$-homogeneous if $\phi(\lam K)=\lam^{k}\phi(K)$
for any $\lam\geq0$ and any convex compact set $K$). McMullen's
decomposition theorem \cite{mcmullen-decomp} says that 
\begin{eqnarray}
Val(\RR^{n})=\oplus_{k=0}^{n}Val_{k}(\RR^{n}).\label{E:mcmullen}
\end{eqnarray}
$Val_{k}(\RR^{n})$ can be decomposed further with respect to parity:
\[
Val_{k}(\RR^{n})=Val_{k}^{ev}(\RR^{n})\oplus Val_{k}^{odd}(\RR^{n}),
\]
where a valuation $\phi$ is called even (resp. odd) if $\phi(-K)=\phi(K)$
(resp. $\phi(-K)=-\phi(K)$) for any $K$.

It is easy to see that $Val_{0}(\RR^{n})$ is spanned by the Euler
characteristic, i.e. valuation which is equal to 1 on any convex compact
set. By a theorem of Hadwiger \cite{hadwiger-book}, $Val_{n}(\RR^{n})$
is spanned by the Lebesgue measure.

We denote by $Val(\RR^{n})^{SO^{+}(n-1,1)}$ the subspace of $SO^{+}(n-1,1)$-invariant
valuations, and similarly for subspaces of given parity and homogeneity.
McMullen's decomposition (\ref{E:mcmullen}) immediately implies 
\[
Val(\RR^{n})^{SO^{+}(n-1,1)}=\oplus_{k=0}^{n}(Val_{k}^{ev}(\RR^{n})^{SO^{+}(n-1,1)}\oplus Val_{k}^{odd}(\RR^{n})^{SO^{+}(n-1,1)})
\]

Our first main result classifies odd $SO^{+}(n-1,1)$-invariant valuations. 
\begin{thm}
For $0\leq k\leq n$, $k\ne n-1$, $Val_{k}^{odd}(\RR^{n})^{SO^{+}(n-1,1)}=0$.
\\
For $k=n-1$, 
\begin{eqnarray*}
\dim Val_{k}^{odd}(\RR^{n})^{SO^{+}(n-1,1)}=\left\{ \begin{array}{cc}
1, & n\geq3\\
2, & n=2
\end{array}\right.
\end{eqnarray*}
The last space will be described explicitly. 
\end{thm}
The proof of this result relies on Schneider's imbedding theorem and
makes use of Lie group continuous cohomology as one of the tools to
show that the Schneider bundle has no non-zero continuous $SO^{+}(n-1,1)$-invariant
sections for $1\leq k\leq n-2$.

\hfill{}

Our second main result classifies even $SO^{+}(n-1,1)$-invariant
valuations. Notice first of all that by the above discussion $0$-
and $n$-homogeneous valuations are proportional to the Euler characteristic
and Lebesgue measure, respectively. In particular they are even and
$SO^{+}(n-1,1)$-invariant.\\
For the remaining degrees of homogeneity, namely $1\leq k\leq n-1$,
the classification consists of two parts. First, we define and classify
the invariant generalized valuations. Roughly speaking, this is the
completion of the space of smooth valuations with respect to a certain
weak topology that is defined using the product structure on valuations
(see subsection \ref{sub:Lorentz-invariant-generalized-valuations}
for precise definitions). The space of generalized valuations naturally
contains the continuous valuations as a dense subspace. We then analyze
which of the invariant generalized valuations are in fact continuous.
The following two theorems summarize our results:
\begin{thm}
For all $1\leq k\leq n-1$, the space of $k$-homogeneous, even, $SO^{+}(n-1,1)$-invariant
generalized valuations is 2-dimensional. Those spaces will be described
explicitly.
\end{thm}

\begin{thm}
For $1\leq k\leq n-2$, $Val_{k}^{ev}(\RR^{n})^{SO^{+}(n-1,1)}=0$.
For $k=n-1$, $\dim Val_{k}^{ev}(\RR^{n})^{SO^{+}(n-1,1)}=2$. This
space will be described explicitly. 
\end{thm}
Let us remark that the generalized Lorentz-invariant odd valuations
remain to be classified.\\
\\
The plan of the classification is as follows: First, we study $SO^{+}(n-1,1)$-invariant
continuous sections of the Klain bundle. For any $1\leq k\leq n-1$
we get a 2-dimensional space of $SO^{+}(n-1,1)$-invariant continuous
sections. By McMullen's theorem, this finishes the classification
of continuous $(n-1)$-homogeneous even valuations. For the remaining
$1\leq k\leq n-2$, it turns out that those sections correspond to
generalized valuations, which are not continuous. We construct the
corresponding generalized valuations explicitly (section \ref{sec:nonexistence of even Lorentz}),
and then proceed to show that they are discontinuous by proving that
they cannot be evaluated on the double cone (sections \ref{sec3: valuations of rotation invariant bodies},\ref{sec:5 real non existence}).
This last part of analysis involves lengthy technical arguments. Another
difficulty in comparison to the case of groups transitive on spheres
is that $SO^{+}(n-1,1)$-invariant valuations do not have to be smooth
in the sense of \cite{alesker-gafa-04}. \\

Finally, we give some applications of the classification. One is the
explicit construction of a continuous section of Klain's bundle that
lies in the closure of Klain's imbedding of the continuous valuations,
but outside the image of the imbedding. The non-closedness of the
image was proved very recently by Parapatits and Wannerer in \cite{Parapatis-Wennerer}
using different methods. Another corollary is the non-extendibility
by continuity of the Fourier transform from smooth to continuous valuations,
which also was proved in \cite{Parapatis-Wennerer} using different
methods. \\

\textbf{Acknowledgement.} We are grateful to José Miguel Figueroa-O'Farrill
who has explained to us Proposition \ref{prop:The-first-continuousCohomologyVanishes}
on computation of continuous group cohomology.

\section{Finding the Lorentz-invariant continuous sections of Klain's and
Schneider's bundles}

Let us introduce the notation used throughout the paper. For a linear
space $W$, $Vol(W)$ will denote the 1-dimensional space of volume
forms on $W$, and $D(W)$ the 1-dimensional space of densities on
$W$. $Gr(W,k)$ is the Grassmannian of $k$-subspaces of $W$. The
signature of a quadratic form $Q$ will be denoted $\sign Q$; We
write $SO^{+}(n-1,1)$ for the identity component of the full Lorentz
group $O(n-1,1)$. If a norm is given on $W$, $S(W)$ denotes the
unit sphere in $W$. For a vector bundle $E$ over a manifold $M$,
$\Gamma^{\pm\infty}(M,E)$, or sometimes simply $\Gamma^{\pm\infty}(E)$,
will denote the space of smooth resp. generalized sections.\\
\\
It is well-known that the even continuous valuations naturally form
a $GL(n)$-equivariant subspace of the continuous sections of Klain's
bundle, which is the line bundle of densities on $k$-dimensional
linear subspaces of $\R^{n}$, over $Gr(n,k)$. A similar result holds
for odd continuous valuations; the precise description is given below.
The definitions here are more technical and will be recalled later.
To find all Lorentz-invariant valuations, we begin by determining
all the invariant sections of those two bundles.\\
\\
In the following, $V$ stands for $\R^{n}$. Fix two symmetric 2-forms:
Euclidean $\langle u,v\rangle=\sum_{j=1}^{n}u_{j}v_{j}$ , and Lorentzian
$Q(u,v)=\sum_{j=1}^{n-1}u_{j}v_{j}-u_{n}v_{n}$. Let $(e_{j})$ be
the standard basis, and $\zeta(v):=\langle v,e_{n}\rangle$. The unit
$n\times n$ matrix is denoted $I$.

\subsection{Klain's bundle $K^{n,k}$}

Let $\gamma_{n}^{k}$ be the tautological vector bundle over $Gr(n,k)$,
so that the fiber over $\Lambda\in Gr(\R^{n},k)$ is simply $\Lambda$;
and $K^{n,k}$ is the bundle of densities on its fibers, which is
naturally a $GL(n)$-line bundle.The Euclidean structure defines a
density in every subspace, i.e. we have a global section $Area\in\Gamma(K^{n,k})$,
$Area$ is the only $SO(n)$-invariant continuous section (up to scale),
and it defines a trivialization of the bundle. Denote by $SO^{+}(n-1,1)\subset GL(n)$
the connected component of the identity in the group of isometries
of $Q$. We will study $SO^{+}(n-1,1)$-invariant continuous sections
of $K^{n,k}$. 
\begin{prop}
\label{prop:Lorentz_area}Given a Lorentz-orthogonal family ($v_{1},...,v_{k})$
s.t. $Q(v_{i})=1$ for $i\leq k-1$, and $Q(v_{k})=\pm1$, and denoting
$z_{j}=\zeta(v_{j}),$ one has
\[
Area(v_{1},...,v_{k})^{2}=\Bigg\{\begin{array}{c}
1+2\sum_{j=1}^{k}z_{j}^{2},\,\, Q(v_{k})=1\\
2(z_{k}^{2}-\sum_{j=1}^{k-1}z_{j}^{2})-1,\,\, Q(v_{k})=-1
\end{array}
\]
\end{prop}
\begin{proof}
Use the identity
\[
Area(v_{1},...,v_{k})^{2}=det(\langle v_{i},v_{j}\rangle)=\det(Q(v_{i},v_{j})+2z_{i}z_{j})=\det(I_{\pm}+2zz^{T})
\]
where $I_{\pm}$ is a $k\times k$ diagonal matrix with entries $Q(v_{1}),...,Q(v_{k})$,
and $z=(z_{1},...,z_{k})^{T}$. The remaining verification is straightforward.\end{proof}
\begin{prop}
\label{prop:NoObstructionToExtend}Given $T\in SO^{+}(n-1,1)$, and
$\Lambda\in Gr(n,k)$ generic (i.e., $Q$ restricted to $\Lambda$
is non-degenerate), if $T(\Lambda)=\Lambda$ then $|\det T|_{\Lambda}|=1$.\end{prop}
\begin{proof}
Since $Q|_{\Lambda}$ is non-degenerate, and $T\in GL(\Lambda)$ preserves
$Q$, it follows that $|\det T|_{\Lambda}|=1$.\end{proof}
\begin{prop}
\label{prop:KlainSectionsAre2Dimensional}The space of $G=SO^{+}(n-1,1)$-invariant
continuous sections of $K^{n,k}$ is $2$-dimensional \end{prop}
\begin{proof}
1. The orbits of the action of $G$ on $Gr(n,k)$ are characterized
by the signature of the restriction of $Q$. The open orbits are $M_{+}=\{\Lambda:\sign Q|_{\Lambda}=(k,0)\}$
and $M_{-}=\{\Lambda:\sign Q|_{\Lambda}=(k-1,1)\}$. Together, $M_{+}\cup M_{-}$
are dense in $Gr(n,k)$. The remaining orbit is $M_{0}=\{\Lambda:signQ|_{\Lambda}=(k-1,0)\}$
(there are no 2 non-proportinal $Q$-orthogonal vectors on the light
cone).

2. Choose some fixed $\Lambda_{+}\in M_{+}$ and $\Lambda_{-}\in M_{-}$,
and fix arbitrary densities on them. By Proposition \ref{prop:NoObstructionToExtend},
they extend to an invariant section $\mu$ of $M_{+}\cup M_{-}$.
It remains to verify that $\mu$ admits a continuous $G-$invariant
extension to all $Gr(n,k)$. Let us show that $\mu(\Lambda)\to0$
as $\Lambda\to M_{0}$. For this, it is enough to take a $Q$-orthonormal
basis of $\Lambda$, denoted $v_{1},...,v_{k}$ and show that $Area(v_{1},...,v_{k})^{2}\to\infty$.

3. First, assume $M_{+}\ni\Lambda\to M_{0}$. 

Write $z=(z_{1},...,z_{k})$. Define $\epsilon$ by $\langle P_{\Lambda}e_{n},e_{n}\rangle=\cos(\pi/4+\epsilon)|P_{\Lambda}e_{n}|$,
where $P_{\Lambda}$ is the (Euclidean) projection onto $\Lambda$.
We assume $Q(v_{j})=1$ for $1\leq j\leq k$, Write $P_{\Lambda}e_{n}=\sum\alpha_{j}v_{j}$.
Then $\langle P_{\Lambda}e_{n}-e_{n},v_{i}\rangle=0,$ for all $i$,
i.e. $(I+2zz^{T})(\alpha)=z$. By Sherman-Morrison \cite{SM} formula,
\[
(I+2zz^{T})^{-1}=I-\frac{2zz^{T}}{1+2z^{T}z}I
\]
We will denote $A=Area(v_{1},...,v_{k})^{2}$, $B=z^{T}z=z_{1}^{2}+...+z_{k-1}^{2}+z_{k}^{2}$.
By Proposition \ref{prop:Lorentz_area}, $A=1+2B$. Then 
\[
\alpha=z-\frac{2zz^{T}z}{1+2z^{T}z}=\frac{1}{A}z
\]
Let us write $\cos^{2}(\pi/4+\epsilon)=1/2-\delta$. Then 
\[
\langle P_{\Lambda}e_{n},e_{n}\rangle^{2}=\cos^{2}(\pi/4+\epsilon)|P_{\Lambda}e_{n}|^{2}\Rightarrow\zeta(P_{\Lambda}e_{n})^{2}=(1/2-\delta)\Big(Q(P_{\Lambda}e_{n})+2\zeta(P_{\Lambda}e_{n})^{2}\Big)\Rightarrow
\]
\[
\Rightarrow(\sum\alpha_{j}z_{j})^{2}=(1/2-\delta)(\sum_{j=1}^{k}\alpha_{j}^{2})+(1-2\delta)(\sum\alpha_{j}z_{j})^{2}
\]

\[
\Rightarrow2\delta(\sum\alpha_{j}z_{j})^{2}=(1/2-\delta)(\sum_{j=1}^{k}\alpha_{j}^{2})
\]
Note that $\sum\alpha_{j}z_{j}=A^{-1}(z_{1}^{2}+...+z_{k-1}^{2}+z_{k}^{2})=\frac{B}{A}=\frac{A-1}{2A}=\frac{1}{2}-\frac{1}{2A}$,
and $\sum_{j=1}^{k}\alpha_{j}^{2}=\frac{B}{A^{2}}=\frac{A-1}{2A^{2}}$
. Thus 
\[
\delta(1-\frac{1}{A})^{2}=(1/2-\delta)\frac{1}{A}(1-\frac{1}{A})\Rightarrow\frac{1}{A}=\frac{\delta}{1/2-\delta}(1-1/A)
\]

\[
\Rightarrow A=\frac{1}{2\delta}=\frac{1}{\sin2\epsilon}
\]
Thus $Area(v_{1},...,v_{k})=\frac{1}{|\sin2\epsilon|^{1/2}}\to\infty$
as $\delta\to0$, i.e. $\mu(\Lambda)\to0$ as $M_{+}\ni\Lambda\to M_{0}$.
This proves the existence of a section supported on $M_{+}$.

4. Now assume $M_{-}\ni\Lambda\to M_{0}$. Write $z=(z_{1},...,z_{k})$.
Let $\langle P_{\Lambda}e_{n},e_{n}\rangle=\cos(\pi/4-\epsilon)|P_{\Lambda}e_{n}|$,
where $P_{\Lambda}$ is the orthogonal projection onto $\Lambda$.
We assume $Q(v_{j})=1$ for $1\leq j\leq k-1$, $Q(v_{k})=-1$. Write
$P_{\Lambda}e_{n}=\sum\alpha_{j}v_{j}$. Then $\langle P_{\Lambda}e_{n}-e_{n},v_{i}\rangle=0,$
for all $i$, i.e. $(I_{-}+2zz^{T})(\alpha)=z$. By Sherman-Morrison,
\[
(I_{-}+2zz^{T})^{-1}=I_{-}-\frac{2I_{-}zz^{T}I_{-}}{1+2z^{T}I_{-}z}
\]
We will denote $\tilde{z}=I_{-}z$. Again using Proposition \ref{prop:Lorentz_area},
we write $B=z^{T}I_{-}z=z_{1}^{2}+...+z_{k-1}^{2}-z_{k}^{2}$, $A=Area(v_{1},...,v_{k})^{2}=-1-2B$.
Then 
\[
\alpha=(I_{-}+2zz^{T})^{-1}z=I_{-}z-\frac{2}{1+2B}I_{-}zz^{T}I_{-}z=\tilde{z}-\frac{2B}{1+2B}I_{-}z=\frac{1}{1+2B}\tilde{z}
\]
That is,
\[
\alpha=-\frac{1}{A}\tilde{z}
\]
Let us write $\cos^{2}(\pi/4-\epsilon)=1/2+\delta$. Then 
\[
\langle P_{\Lambda}e_{n},e_{n}\rangle^{2}=\cos^{2}(\pi/4-\epsilon)|P_{\Lambda}e_{n}|^{2}\Rightarrow\zeta(P_{\Lambda}e_{n})^{2}=(1/2+\delta)\Big(Q(P_{\Lambda}e_{n})+2\zeta(P_{\Lambda}e_{n})^{2}\Big)\Rightarrow
\]
\[
\Rightarrow(\sum\alpha_{j}z_{j})^{2}=(1/2+\delta)(\sum_{j=1}^{k-1}\alpha_{j}^{2}-\alpha_{k}^{2})+(1+2\delta)(\sum\alpha_{j}z_{j})^{2}
\]

\[
\Rightarrow-2\delta(\sum\alpha_{j}z_{j})^{2}=(1/2+\delta)(\sum_{j=1}^{k-1}\alpha_{j}^{2}-\alpha_{k}^{2})
\]
Note that $\sum\alpha_{j}z_{j}=-A^{-1}(z_{1}^{2}+...+z_{k-1}^{2}-z_{k}^{2})=-\frac{B}{A}=\frac{A+1}{2A}=\frac{1}{2}+\frac{1}{2A}$,
and $\sum_{j=1}^{k-1}\alpha_{j}^{2}-\alpha_{k}^{2}=\frac{B}{A^{2}}=-\frac{A+1}{2A^{2}}$
. Thus 
\[
\delta(1+\frac{1}{A})^{2}=(1/2+\delta)\frac{1}{A}(1+\frac{1}{A})\Rightarrow\frac{1}{A}=\frac{\delta}{1/2+\delta}(1+1/A)
\]

\[
\Rightarrow A=\frac{1}{2\delta}=\frac{1}{\sin2\epsilon}
\]
Again $Area(v_{1},...,v_{k})=\frac{1}{|\sin2\epsilon|^{1/2}}\to\infty$
as $\delta\to0$, i.e. $\mu(\Lambda)\to0$ as $M_{-}\ni\Lambda\to M_{0}$.
Thus there is an invariant section supported on $M_{-}$,\end{proof}
\begin{cor}
\label{cor:Lorentz_surface_area}The space $Val_{n-1}^{ev}(\R^{n})^{SO^{+}(n-1,1)}$
is 2-dimensional, and consists of non-smooth sections. It is spanned
by $f_{S}$ and $f_{T}$ (standing for space-like and time-like) given
by 
\[
f_{T}(K)=\int_{S^{n-1}\cap\{Q\geq0\}}\sqrt{|\sin2\epsilon|}d\sigma_{K}(\omega)
\]
and similarly
\[
f_{S}(K)=\int_{S^{n-1}\cap\{Q\leq0\}}\sqrt{|\sin2\epsilon|}d\sigma_{K}(\omega)
\]
where $\epsilon$ denotes the angle between $\omega$ and the light
cone, and $\sigma_{K}(\omega)$ is the surface area measure of $K$.
\end{cor}

\subsubsection{Geometrical interpretations of the space $Val_{n-1}^{ev}(\R^{n})^{SO^{+}(n-1,1)}$}

This purpose of this subsection is to provide some geometrical intuition
into the valuations that we constructed. It will not be used in the
rest of the paper. 

Let us denote $H^{\pm}=\{x\in\R^{n}:Q(x,x)=\pm1\}$. Both $H^{+}$
and $H^{-}$ inherit a Lorentzian resp. Riemannian metric from $(\R^{n},Q)$.
Then $H^{-}\subset(\R^{n},Q)$ is the Minkowski model of hyperbolic
space, and similarly $H^{+}$ is the $(n-2,1)$ de Sitter space. The
valuations in $Val_{n-1}^{ev}(\R^{n})^{SO^{+}(n-1,1)}$ can be interpreted
as the surface area of $K$ with respect to $H^{\pm}$ in the following
sense:

Define the support functions $h_{H^{+}},h_{H^{-}}:S^{n-1}\to\R$ by
setting $h_{H^{\pm}}(\theta)$ equal to the distance from the origin
of the hyperplane $P_{\theta}$ with Euclidean normal equal to $\theta$
that is tangent to $H^{+}$ (resp. $H^{-}$). If no such hyperplane
exists, the value of $h_{H^{\pm}}(\theta)$ is set to $0$. Denoting
by $-\frac{\pi}{2}\leq\alpha\leq\frac{\pi}{2}$ the elevation angle
on $S^{n-1}$ relative to the spacelike coordinate hyperplane $(x_{1},...,x_{n-1})$,
these functions are given explicitly by 
\[
h_{H^{+}}(\omega)=\Bigg\{\begin{array}{cc}
\sqrt{|\cos2\alpha|} & |\alpha|\leq\pi/4\\
0 & |\alpha|<\pi/4
\end{array}
\]
\[
h_{H^{-}}(\omega)=\Bigg\{\begin{array}{cc}
\sqrt{|\cos2\alpha|} & |\alpha|\geq\pi/4\\
0 & |\alpha|<\pi/4
\end{array}
\]
Then we may think of $f_{T}$ informally as a mixed volume: 
\[
f_{T}(K)=V(K[n-1],H^{+}[1])=\int_{S^{n-1}}h_{H^{+}}(\omega)d\sigma_{K}(\omega)
\]
and similarly 
\[
f_{S}(K)=V(K[n-1],H^{-}[1])=\int_{S^{n-1}}h_{H^{-}}(\omega)d\sigma_{K}(\omega)
\]
\\
Another very similar description is the following. Assume $K$ is
smooth. The boundary $\partial K$ inherits from $(\R^{n},Q)$ a smooth
field of quadratic forms on all tangent spaces. Then $f_{S}(K)$ is
the total volume of the space-like part of $\partial K$ (that is,
where the form is positively defined), and similarly $f_{T}(K)$ is
the volume of the time-like part. \\
\\
There is also a relation between the $(n-1)$-homogeneous Lorentz-invariant
valuations, and the surface area in hyperbolic and de Sitter spaces.
More precisely, $f_{T}$ and $f_{S}$ correspond to the surface area
on $H^{-}$ and $H^{+}$, respectively, in the following sense. For
a set $A\subset H^{\pm}$, define $C_{A}=\{tx:0\leq t\leq1,x\in A\}$
the cone with base $A$. Denote by $Area_{H^{\mp}}$ the hyperbolic\textbackslash{}de
Sitter area on $H^{\mp}$, and $\phi_{H}^{\mp}$ is either $f_{T}$
or $f_{S}$, respectively. Observe that while $C_{A}$ is not a convex
body, one can nevertheless compute $f_{S}$ or $f_{T}$ on $C_{A}$
at least when $A$ is piecewise geodesic (and so given by a finite
collection of intersections of $H^{\pm}$ with hyperplanes in $\R^{n}$),
simply by applying the explicit formulas of Corollary \ref{cor:Lorentz_surface_area}.
\begin{prop}
Let $A\subset H^{\pm}$ be a polytope. If $A\subset H^{+}$, we further
assume it has spacelike boundary. Then 
\[
Area_{H^{\pm}}(\partial A)=\phi_{H}^{\pm}(C_{A})
\]
\end{prop}
\begin{proof}
An $(n-2)$-dimensional face $F$ of $A$ lies on $\Lambda\cap H^{\pm}$
for $\Lambda\in Gr(n,n-1)$. By additivity of both sides, it suffices
to verify that $Area_{H^{\pm}}(F)=\phi_{H}^{\pm}(C_{F})$. For $H^{+}$,
by our assumption $\Lambda$ is space-like, so the statement is simply
that the cone measure on the sphere $\Lambda\cap H^{+}$ coincides
with the spherical volume on it. For $H^{-}$, $\Lambda$ is necessarily
timelike, and it is again well-known (or easily checked) that the
cone measure of the hyperboloid $\Lambda\cap H^{-}$ coincides with
the hyperbolic volume.
\end{proof}

\subsection{Schneider's bundle $S^{n,k}$}

For every non-oriented subspace $\Omega\subset V$ of dimension $k+1$,
consider the bundle of densities on the tautological bundle over the
space of $k$-dimensional cooriented subspaces $\Lambda\subset\Omega$,
denoted $\widetilde{K}^{k+1,k}(\Omega)$. Let $\Gamma_{odd}(\widetilde{K}^{k+1,k}(\Omega))$
denote the space of global sections which are odd w.r.t. coorientation
reversal of $\Lambda$, i.e. $\mu(\overline{\Lambda})=-\mu(\Lambda)$.
\\
\\
There is a $k+1$-dimensional linear subspace $L(\Omega)\subset\Gamma_{odd}(\widetilde{K}^{k+1,k}(\Omega))$,
consisting of sections that are defined by elements of $\Omega^{*}$.
The space $L(\Omega)\simeq\Omega\otimes D(\Omega)$ is defined as
follows: For any $k$-dimensional $\Lambda\subset\Omega$, $\Lambda^{\perp}\subset\Omega^{*}$,
and 
\[
D(\Lambda)\otimes D(\Omega/\Lambda)=D(\Omega)
\]
so 
\[
D(\Lambda)=D(\Omega)\otimes D(\Omega/\Lambda)^{*}=D(\Omega)\otimes D(\Lambda^{\perp})
\]
Any $v\in\Omega$ defines a density $|v|$ on $\Lambda^{\perp}\subset\Omega^{*}$
for all $\Lambda$, and so we define $\mu_{v\otimes d}(\Lambda)=\text{sign}(v,\Lambda)|v|\otimes d\in\Gamma_{odd}(\widetilde{K}^{k+1,k}(\Omega))$
for $v\otimes d\in\Omega\otimes D(\Omega)$. Here $\text{sign}(v,\Lambda)=\pm1$
is determined by the coorientation of $\Lambda$ (and $\text{sign}(v,\Lambda)=0$
for $v\in\Lambda$). The image of the map $v\otimes d\mapsto\mu_{v\otimes d}$
is denoted. Let $F_{\Omega}=\Gamma_{odd}(\widetilde{K}^{k+1,k}(\Omega))/L(\Omega)$
be the quotient.\\
\\
Schneider's bundle $S^{n,k}$ consists of the base space $Gr(V,k+1)$,
and fiber $F_{\Omega}$. The topology can be introduced by fixing
an orthonormal basis on $V$, which gives the identifications $\Gamma_{odd}(\widetilde{K}^{k+1,k}(\Omega))=C_{odd}(S(\Omega))$,
$L(\Omega)=\Omega^{*}=\Omega$, and $F_{\Omega}=\Omega^{\perp}\subset C_{odd}(S(\Omega))$,
the orthogonal complement taken in the $L_{odd}^{2}(S(\Omega))$ norm
. In particular, $F_{\Omega}$ inherits an inner product induced from
$L_{odd}^{2}(S(\Omega))$. Note that any global section $s\in\Gamma(S^{n,k})$
gives a continuous in $\Omega$ family of functions $\mu_{\Omega}\in C_{odd}(S(\Omega))$.
Schneider's imbedding gives for every odd $k$-homogeneous valuation
a global section $s\in\Gamma(S^{n,k})$.\\
\\
However, for a $G$-equivariant section $s$ (where $G$ is some group),
this lift is not a-priori $G$-equivariant. This is because the lift
is defined by an arbitrarily chosen Euclidean structure. \\
\\
We will classify the $G=SO^{+}(n-1,1)$-invariant sections of $S^{n,k}$. 
\begin{thm}
There are no odd $SO^{+}(n-1,1)$-invariant $k$-homogeneous valuations
for $1\leq k\leq n-2$. For $k=n-1$ and $n\geq3$, the space $Val_{k}^{-}(\R^{n})^{SO^{+}(n-1,1)}$
is 1-dimensional. Finally, the space $Val_{1}^{-}(\R^{2})^{SO^{+}(1,1)}$
is $2$-dimensional.\end{thm}
\begin{proof}
Let $s$ be such a section. We assume at first that $k\leq n-2$.
\\
\\
0. Denote $M_{+}=\{\Omega\in Gr(V,k+1):Q|_{\Omega}>0\}$, $M_{-}=\{\Omega\in Gr(V,k+1):\sign Q|_{\Omega}=(k,1)\}$,
$M_{0}=\{\Omega\in Gr(V,k+1):\sign Q|_{\Omega}=(k,0)\}$. Those are
the orbits of $G$ as it acts on $Gr(V,k+1)$. We will write $\text{Stab}(\Omega)\subset G$
for the stabilizer of $\Omega$, and $\Stab^{+}(\Omega)=\{T\in\Stab(\Omega):\det T|_{\Omega}=1\}$
is the orientation-preserving subgroup of $\Stab(\Omega)$. \\
\\
1. Observe that $s$ necessarily vanishes on $M_{+}$: Fix some $\Omega\in M_{+}$.
Take the Euclidean structure on $\Omega$ to be $Q|_{\Omega}$, and
then obtain a lift $\mu_{\Omega}\in\Gamma_{odd}(\widetilde{K}^{k+1,k}(\Omega))$
of $s_{\Omega}$ which is $\text{Stab}(\Omega)$-invariant. Since
$\text{Stab}(\Omega)$ is transitive on $Gr^{+}(\Omega,k)$ (in fact,
it is transitive even under $\Stab^{+}(\Omega)$), $\mu_{\Omega}(\Lambda)=\mu_{\Omega}(\overline{\Lambda})$
for all $\Lambda$, so $\mu_{\Omega}=0$ on $\Omega$. Thus $s=0$
on $M_{+}$, and by continuity, it follows that $s$ vanishes on $M_{0}$.\\
\\
2. Now consider $M_{-}$. For any fixed $\Omega\in M_{-}$ one has
a $\text{Stab}(\Omega)$-invariant element $s_{\Omega}\in\Gamma_{odd}(\widetilde{K}^{k+1,k}(\Omega))/L(\Omega)$
. Since $H_{c}^{1}(\Stab^{+}(\Omega);\Omega)=0$ (see \ref{sub:Computation of Lie algebra cohomology}
below for the computation), we can choose $\mu_{\Omega}\in\Gamma_{odd}(\widetilde{K}^{k+1,k}(\Omega))^{\Stab^{+}(\Omega)}$
lifting $s_{\Omega}$, and then, possibly after averaging with $g_{0}\mu_{\Omega}$
(which also lifts $s_{\Omega})$ where $g_{0}\in\Stab(\Omega)$ is
orientation-reversing, we may assume $\mu_{\Omega}\in\Gamma_{odd}(\widetilde{K}^{k+1,k}(\Omega))^{\Stab(\Omega)}$.
In fact, if we fix any $\Omega_{0}\in M_{-}$ and the corresponding
$\mu_{0}=\mu_{\Omega_{0}}$, then for any $g\in G$ one can take $\mu_{g\Omega_{0}}=g_{*}\mu_{0}\in\Gamma_{odd}(\widetilde{K}^{k+1,k}(g\Omega_{0}))^{\Stab(g\Omega_{0})}$.
We thus get a $G$-invariant lift of $s$ to a continuous family of
sections of $\Gamma_{odd}(\widetilde{K}^{k+1,k}(\Omega))$ over $\Omega\in M_{-}$.
\\
\\
3. We want to inspect $\mu_{0}$ more closely. The group $\text{Stab}(\Omega_{0})$
has the following open orbits as it acts on the cooriented hyperplanes
$\Lambda\subset\Omega_{0}$: Ignoring the coorientation, there are
two non-oriented open orbits, consisting of $X_{+}$, the $Q$-positive
$\Lambda$ and $X_{-}$, those $\Lambda$ with signature $(k-1,1)$.
\\
An orientation of $\Lambda\in X_{+}$ is fixed under $g\in\text{Stab}(\Omega_{0})\cap\text{Stab}(\Lambda)$
iff the orientation of $\Omega_{0}$ is fixed, so coorientation is
always preserved. Thus $X_{+}$ splits into two orbits $X_{1}$ and
$X_{2}$ when coorientation is accounted for. \\
On the other hand, $X_{-}$ constitutes a single orbit including coorientation.
There are two cases to consider: when $k=1$, $\Lambda\in X_{-}$
is a time-like line and so has its orientation preserved under the
action of $g\in\text{Stab}(\Omega_{0})\cap\text{Stab}(\Lambda)$,
while the orientation of $\Omega_{0}$ can be preserved or reversed
(since $\dim\Omega_{0}=k+1\leq n-1$). \\
If $k\geq2$, the verification is also straightforward: one can again
reverse the orientation of $\Omega_{0}$ while kee ping the orientation
of $\Lambda$.\\
We conclude that $\mu_{0}(\Lambda)=0$ for all $\Lambda\in X_{-}$:
Indeed, since $\mu_{0}$ is odd, $\mu_{0}(\overline{\Lambda})=-\mu_{0}(\Lambda)$;
but both $\Lambda,\overline{\Lambda}$ lie in the same $\text{Stab}(\Omega_{0})$-orbit,
so $\mu_{0}(\Lambda)=0$. \\
\\
4. Observe that on any $\Lambda\subset\Omega_{0}$ which is $Q$-degenerate,
$\mu_{0}(\Lambda)=0$ by continuity from $X_{+}$. \\
So $\mu_{0}$ is uniquely defined (since it is odd, and through $G$-invariance)
by a density $\mu_{+}\in D(\Lambda_{+})$ for some $Q$ - positive
subspace $\Lambda_{+}\subset\Omega_{0}$.\\
Note that, as was the case with Klain's bundle, any such $\mu_{+}$
extends to a continuous $\mu_{0}\in\Gamma_{odd}(\widetilde{K}^{k+1,k}(\Omega_{0}))^{\Stab(\Omega_{0})}$,
and then to a family $\mu_{\Omega}$ for $\Omega\in M_{-}$ \\
\\
5. Let us show that $\mu_{\Omega}$ has a limit $\mu_{\infty}$ in
$\Gamma_{odd}\left(\widetilde{K}^{k+1,k}(\Omega_{\infty})\right)$
as $\Omega\to\Omega_{\infty}\in M_{0}$. Assume for simplicity that
some orientation is fixed on $\Omega_{\infty}$. For every $Q$-positive
oriented $k$-subspace $\Lambda\subset V$, choose $\Omega_{\Lambda}=\Lambda\oplus\langle e_{n}\rangle$
with the natural orientation, and $\mu(\Lambda)=\mu_{\Omega_{\Lambda}}(\Lambda)\in D(\Lambda)$.
The family $\mu_{\Omega}$ is thus equivalent to a $G$-equivariant
collection $\mu(\Lambda)$ of densities on all $Q$-positive $k$-dimensional
oriented subspaces $\Lambda$, s.t. $\mu(\overline{\Lambda})=-\mu(\Lambda)$.
Then for $M_{-}\ni(\Omega_{t},\Lambda_{t})\to(\Omega_{\infty},\Lambda_{\infty})$,
either $\mu(\Lambda_{t})\to\mu(\Lambda_{\infty})$ when $\Lambda_{\infty}$
is $Q$-positive by continuity of $\mu_{\Omega}$, or $\mu(\Lambda_{t})\to0\in D(\Lambda_{\infty})$.
Thus $\mu_{\infty}$ is well-defined. The limit of $[\mu_{\Omega}]$
in $\Gamma_{odd}\left(\widetilde{K}^{k+1,k}(\Omega_{\infty})\right)/L(\Omega_{\infty})$
is therefore $[\mu_{\infty}]$, and it must vanish as $\Omega\to\Omega_{\infty}\in M_{0}$,
by continuity of $s$ and since $s$ vanishes on $M_{+}$. Therefore,
$\mu_{\infty}$ is a linear section that vanishes on all $Q$-degenerate
$k$-subspaces. This is equivalent to a linear functional on $\R^{k+1}$
that vanishes on the light cone. So $\mu_{\infty}=0$, implying $\mu_{\Omega}=0$.\\
\\
We conclude that when $k\leq n-2$, there are no $G$-invariant sections
of Schneider's bundle. It follows there are no non-trivial continuous,
odd, $k-$homogeneous valuations.

Now assume $k=n-1$. Again since $H_{c}^{1}(G;V)=0$, we may lift
$s$ to an invariant section $\mu\in\Gamma_{odd}(\widetilde{K}^{n,n-1}(V))^{G}$.\\
If $n\geq3$, as in step $3$ above, $\mu$ must vanish on mixed-signature
subspaces; and $\mu$ is determined by its value $\mu_{+}$ on one
positive subspace. Unlike the case $k\leq n-2$, there are no other
restrictions: any $\mu_{+}$ extends to a global section $\mu$, as
was the case with Klain's bundle.\\
If $n=2$, as in step 2 above $\mu$ is determined by two independent
densities $\mu_{+}(\Lambda_{+})$ and $\mu_{-}(\Lambda_{-})$; and
any two such densities give a continuous $\mu_{\Omega}$ as with Klain's
bundle. \\
\\
For $k=n-1$, Schneider's imbedding is really just the McMullen characterization
of odd $n-1$-homogeneous valuations, i.e. the imbedding is an isomorphism,
concluding the classification of $n-1$ -homogeneous invariant valuations.
\end{proof}

\subsubsection{\label{sub:Computation of Lie algebra cohomology}Computation of
the continuous Lie group cohomology}

The main result of this section was explained to us by José Miguel
Figueroa-O'Farrill. For the relevant definitions, see \cite{Fuks}.
We need to compute the continuous cohomolgy of $G=SO^{+}(n-1,1)$
with coefficients in the standard representation $V=\R^{n}$. Specifically,
we will show 
\begin{prop}
\label{prop:The-first-continuousCohomologyVanishes}The first continuous
group cohomology $H_{c}^{1}(G;V)$ vanishes. \end{prop}
\begin{proof}
Consider $SO(n-1)\subset G$ - the maximal compact subgroup. By the
Hochschild-Mostow Theorem , 
\[
H_{c}^{1}(G;V)=H^{1}(\mathfrak{so}(n-1,1),\mathfrak{so}(n-1);V)
\]
We will write $\mathfrak{g}=\mathfrak{so}(n-1,1)$ and $\mathfrak{h}=\mathfrak{so}(n-1)$.
Under the action of $\mathfrak{h}$, $V=W\oplus T$ where $W=\R^{n-1}$
is the standard representation of $SO(n-1)$ (corresponding to the
space coordinate hyperplane), and $T=\R$ is the trivial representation
(corresponding to the time axis of $V$). Also, the adjoint action
of $\mathfrak{h}$ on $\mathfrak{g}$ admits the decomposition $\mathfrak{g}=\mathfrak{h}\oplus W$
where the inclusion $i:W\hookrightarrow\mathfrak{g}$ is given by
\[
v\mapsto\Bigg(\begin{array}{cc}
0_{(n-1)\times(n-1)} & v_{(n-1)\times1}\\
v_{1\times(n-1)}^{T} & 0
\end{array}\Bigg)
\]
Note also that $[\mathfrak{h},W]=W$. Now 
\[
C^{0}(\mathfrak{g},\mathfrak{h};V)=\{v\in V:\mathfrak{h}v=0\}=T=\R
\]
while 
\[
C^{1}(\mathfrak{g},\mathfrak{h};V)=\{f\in\text{Hom}(\mathfrak{g},V):f(\mathfrak{h})=0,f([h,g])=hf(g)\,\forall g\in\mathfrak{g},h\in\mathfrak{h}\}=
\]
\[
=\{f\in\text{Hom}(W,V):,f([h,w])=hf(w)\,\forall w\in W,h\in\mathfrak{h}\}=
\]
\[
=\{f\in\text{Hom}(W,W):,f([h,w])=hf(w)\,\forall w\in W,h\in\mathfrak{h}\}
\]
that is, $C^{1}(\mathfrak{g},\mathfrak{h};V)=Hom(W,W)^{\mathfrak{h}}$.
This space consists of scalar operators when $\dim W\geq3\iff n\geq4$,
and of complex-linear operators when $n=3$ and $W=\R^{2}=\C$. The
differential map $d_{1}:C^{0}(\mathfrak{g},\mathfrak{h};V)\to C^{1}(\mathfrak{g},\mathfrak{h};V)$
is nonzero: taking some $t\in T$, $d_{1}t(w)=-i(w)(t)=-tw$ so $d_{1}t\neq0$.
Thus $\dim\text{Im}(d_{1})=1$.\\
For $n\geq4$, $\dim C^{1}(\mathfrak{g},\mathfrak{h};V)=1$ and it
follows that $H^{1}(\mathfrak{g},\mathfrak{h};V)=0$. \\
\\
When $n=3$, $\dim C^{1}(\mathfrak{g},\mathfrak{h};V)=2$ while $d_{1}(C^{0}(\mathfrak{g},\mathfrak{h};V))\subset\text{Ker}(d_{2})\subset C^{1}(\mathfrak{g},\mathfrak{h};V)$.
We should check whether $d_{2}=0$. It is enough to check the value
of $d_{2}$ on some non-scalar operator, say $J\in\text{Hom}(W,W)^{\mathfrak{h}}$
which corresponds to $\frac{\pi}{2}$-rotation. Let $w_{1}$, $w_{2}$
be the standard basis of $W$. Then 
\[
d_{2}J(g_{1},g_{2})=J([g_{1},g_{2}])-g_{1}J(g_{2})+g_{2}J(g_{1})
\]
Since $\mathfrak{h}\subset\text{Hom}(d_{2}J)$ and $\mathfrak{g}=\mathfrak{h}\oplus W$,
$d_{2}J\neq0\iff d_{2}J(w_{1},w_{2})\neq0$. Now 
\[
[i(w_{1}),i(w_{2})]=J\in\mathfrak{h}
\]
so $J([i(w_{1}),i(w_{2})])=0$. And 
\[
-i(w_{1})J(i(w_{2}))+i(w_{2})J(i(w_{1}))=i(w_{1})w_{1}+i(w_{2})w_{2}=(0,0,2)^{T}
\]
so $d_{2}J\neq0$.\\
\\
Thus $\dim\text{Ker}d_{2}=1$ also for $n=3$, and $H^{1}(\mathfrak{so}(n-1,1),\mathfrak{so}(n-1);V)=0$
for all $n$.
\end{proof}
Now consider the exact sequence $0\to L(V)\to\Gamma_{odd}(\widetilde{K}^{n,n-1}(V))\to F_{V}\to0$
where $L(V)$ is the space of linear sections on $V$ (an $n$-dimensional
space), and it is $G$-isomorphic to $V$. We have the long exact
sequence of cohomology 
\[
0\to L(V)^{G}\to\Gamma_{odd}(\widetilde{K}^{n,n-1}(V))^{G}\to F_{V}^{G}\to H^{1}(G;L(V))=0
\]
it follows that every $G$-invariant section of $F_{V}$ lifts to
a $G$-invariant section of $\Gamma_{odd}(\widetilde{K}^{n,n-1}(V))$.

\section{\label{sec3: valuations of rotation invariant bodies}Computing valuations
on $SO(n-1)$-invariant unconditional bodies}
\begin{defn}
The $k$-support function of a body $K\subset\R^{n}$, denoted $h_{k}(\Lambda;K)\in C(Gr(n,n-k))$,
is the $k$-volume of the projection of $K$ to $\Lambda^{\perp}$.
\end{defn}
Let $L\subset\R^{2}$ be a convex, unconditional body. Denote $L^{n}\subset\R^{n}$
its rotation body around the vertical axis, namely 
\[
L^{n}=\{(\omega x,y)|\omega\in S^{n-2},(x,y)\in L\}
\]
Denote also $h_{k}(\alpha;L)=h_{k}(\alpha;L^{k+1})$ for $-\frac{\pi}{2}<\alpha<\frac{\pi}{2}$:
it is obvious that the $k$-support function of $L^{n}$ is $SO(n-1)$-invariant
for all $1\leq k\leq n-1$, and so it really is a function of $\alpha$.
Here $\alpha=0$ corresponds to a vertical hyperplane. By abuse of
notation, we consider $h_{k}(\alpha;L)$ to be a function both on
the unit circle $S^{1}$ and on the sphere $S^{k}\subset\R^{k+1}$;
we will write $h_{k}(\alpha)$ or $h_{k}(\omega)$ when we need to
emphasize that the domain is $S^{1}$, resp. $S^{k}$. Denoting $R_{k+1}\in O(k+1)$
the reversal of time direction and $G_{k+1}=\langle SO(k),R_{k+1}\rangle\subset O(k+1)$,
it is obvious that $L^{n}$ is $G_{n}$-invariant.
\begin{prop}
$L^{n}$ is a convex unconditional body, and $h_{k}(\alpha;L^{n})=h_{k}(\alpha,L)$
for all $n>k$. Any $G_{n}$-invariant convex body equals $L^{n}$
for some $L$ as above.\end{prop}
\begin{proof}
The Minkowski functional of $L^{n}$ is $p_{n}(\omega x,y)=\|(x,y)\|{}_{L}$
for $x,y\in\R$, $\omega\in S^{n-1}$. Let us verify it is convex:
\[
p_{n}(\omega_{1}x_{1},y_{1})+p_{n}(\omega_{2}x_{2},y_{2})=\|(x_{1},y_{1})\|{}_{L}+\|(x_{2},y_{2})\|{}_{L}\geq\|(|x_{1}|+|x_{2}|,|y_{1}|+|y_{2}|)\|_{L}
\]
while 
\[
p_{n}((\omega_{1}x_{1},y_{1})+(\omega_{2}x_{2},y_{2}))=p_{n}(\omega_{1}x_{1}+\omega_{2}x_{2},y_{1}+y_{2})=
\]
\[
=\|(|\omega_{1}x_{1}+\omega_{2}x_{2}|,y_{1}+y_{2})\|_{L}\leq\|(|x_{1}|+|x_{2}|,|y_{1}|+|y_{2}|)\|_{L}
\]
by unconditionality of $L$. The unconditionality of $L^{n}$ is obvious.
Now $h_{k}(\alpha;L^{n})$ can be computed as follows. Let $e_{1},...,e_{n}$
be the standard basis, and define $\Omega=\text{Span}\{e_{1},...,e_{k},e_{n}\}$.
Let $\Lambda_{\alpha}\subset\Omega$ be a $k$-dimensional subspace
forming angle $\alpha$ with the spacelike coordinate hyperplane.
Then $h_{k}(\alpha;L^{k+1})=h_{k}(\alpha;\Omega\cap L^{n})=\text{vol}_{k}(\text{Pr}{}_{\Lambda_{\alpha}}(\Omega\cap L^{n}))$
and by unconditionality of $L$, $\text{Pr}{}_{\Omega}(L^{n})=L^{n}\cap\Omega$
so 
\[
h_{k}(\alpha,L^{n})=\text{vol}_{k}(\text{Pr}{}_{\Lambda_{\alpha}}(L^{n}))=\text{vol}_{k}(\text{Pr}{}_{\Lambda_{\alpha}}\text{Pr}{}_{\Omega}(L^{n}))=h_{k}(\alpha;L^{k+1})
\]
Finally, given a $G_{n}$- invariant convex body $K$, it is immediate
that its 2-dimensional $x_{1}$-$x_{n}$ section $L$ will be convex
and unconditional, and $K=L^{n}$, concluding the proof.\end{proof}
\begin{rem}
It follows that $L\mapsto L^{n}$ is a Hausdorff homeomorphism between
the spaces of $2$-dimensional unconditional convex bodies and $SO(n-1)$-invariant,
unconditional convex bodies. 
\end{rem}
Recall the cosine transform $T_{k}:C^{\infty}(S^{k})\to C^{\infty}(S^{k})$
given by 
\[
T_{k}(f)(y)=\int_{S^{k}}f(x)|\langle x,y\rangle|dx
\]
is a self-adjoint isomorphism when restricted to even functions, and
extends to an isomorphism of generalized even functions. It is well-known
that $T_{k}(\sigma_{k}(\omega;L))=h_{k}(\omega;L)$ where $\sigma_{k}\in C(S^{k})^{*}$
is the surface-area measure of $L^{k+1}$.
\begin{lem}
\label{smooth_pullback}If $f\in C^{\infty}(\R)$ is even, then $f(|x|)\in C^{\infty}(\R^{n})$.\end{lem}
\begin{proof}
This is because $f(x)=g(x^{2})$ for $g\in C^{\infty}[0,\infty)$.
\end{proof}
For the following, we recall the definition of Sobolev spaces. On
the linear space $\R^{k}$, denote $f\mapsto\hat{f}$ the Fourier
transform, and the $p$-Sobolev space is the completion of $C_{c}^{\infty}(\R^{k})$
w.r.t. the norm $\|f\|_{L_{p}^{2}}=\|\hat{f}(\omega)(1+|\omega|^{p})\|_{L^{2}}$.
For a compact smooth manifold $X$, $L_{p}^{2}(X)\subset C^{-\infty}(X)$
is defined by some choice of a finite atlas $\{U_{\alpha}\}$ for
$X$ and an attached partition of unity $\{\rho_{\alpha}\}$:
\[
L_{p}^{2}(X)=\{\sum_{\alpha}\rho_{\alpha}f_{\alpha}:f_{\alpha}\in L_{p}^{2}(U_{\alpha})\}
\]
The resulting space $L_{p}^{2}(X)$ is independent of the choices
made.
\begin{prop}
\label{prop:smoothnessOfSupportFnction}For all $k\geq1$ and $\epsilon>0$,
$h_{k}(\omega;L)\in L_{\frac{3}{2}-\epsilon}^{2}(S^{k})$. If $h_{1}(\alpha;L)$
is smooth in a neighborhood of the poles and the equator, then $h_{k}(\omega;L)\in L_{\frac{k}{2}+1-\epsilon}^{2}(S^{k})$
is smooth near the poles, and $h_{k}(\alpha;L)\in L_{\frac{k}{2}+1-\epsilon}^{2}(S^{1})$.\end{prop}
\begin{proof}
Denote 
\[
\square=\frac{1}{2\omega_{k-1}}(\Delta+k):C_{even}^{\infty}(S^{k})\to C_{even}^{\infty}(S^{k})
\]
where $\omega_{k-1}$ is the surface area of $S^{k-1}$. It is an
invertible differential operator of order 2. Let $\mathcal{R}_{k}:C_{even}^{\infty}(S^{k})\to C_{even}^{\infty}(S^{k})$
denote the spherical Radon transform, which is an invertible Fourier
integral operator of order $-\frac{k-1}{2}$ (see \cite{Radon}).
Then (see \cite{Cosine Transform}) 
\begin{equation}
\square T_{k}=\mathcal{R}_{k}\iff T_{k}=\square^{-1}\mathcal{R}_{k}\label{eq:Cosine Transform}
\end{equation}
Therefore, the cosine transform $T_{k}$ is an invertible (on even
functions) Fourier integral operator of order $-\frac{k+3}{2}$, and
it respects Sobolev spaces, i.e. for all $s\in\R$ 
\[
T_{k}:L_{s}^{2}(S^{k})\to L_{s+\frac{k+3}{2}}^{2}(S^{k})
\]
is an isomorphim. In particular, $T_{1}$ is invertible by a differential
operator followed by a $\frac{\pi}{2}$-rotation.\\
\\
For the first part, note that the surface area measure $\sigma_{k}\in C(S^{k})^{*}\subset L_{-\frac{k}{2}-\epsilon}^{2}(S^{k})$,
so $h_{k}(\omega;L)=T_{k}(\sigma_{k})\in L_{\frac{3}{2}-\epsilon}^{2}(S^{k})$.
\\
\\
For the second part, note that $\sigma_{1}=T_{1}^{-1}(h_{1})\in C(S^{1})^{*}\subset L_{-\frac{1}{2}-\epsilon}^{2}(S^{1})$
is smooth in a neighborhood of the equator and of the poles of $S^{1}$,
since $h_{1}$ is smooth there, and by eq. \ref{eq:Cosine Transform}.
Let $\sigma_{k}=\pi^{*}\sigma_{1}$ be the surface area measure of
$L^{k+1}$, where $\pi:S^{k}\to S^{k}/SO(k-1)$. Then $\sigma_{k}$
is smooth near the poles from unconditionality of $L$ and Lemma \ref{smooth_pullback},
so $\sigma_{k}\in L_{-\frac{1}{2}-\epsilon}^{2}(S^{k})$; also $\sigma_{k}$
is smooth near the equator $S^{k-1}\subset S^{k}$. Therefore, $h_{k}(\omega;L)=T_{k}(\sigma_{k})\in L_{d}^{2}(S^{k})$
where $d=-\frac{1}{2}-\epsilon+\frac{k+3}{2}=\frac{k}{2}+1-\epsilon$,
and also $h_{k}$ is smooth near the poles. Then $h_{k}(\alpha;L)$,
which can be obtained by taking a vertical $2$-dimensional restriction
of $h_{k}(\omega;L)$, lies in $L_{d}^{2}(S^{1})$ and is smooth near
the poles, as required.\end{proof}
\begin{rem}
\label{Rem:Cm_convergence}It follows that under the assumptions of
Proposition \ref{prop:smoothnessOfSupportFnction}, $h_{k}\in C^{\lfloor\frac{k}{2}\rfloor}(S^{1})$,
and if $K_{n}\to K$ in the Hausdorff topology s.t. $h_{1}(\bullet;K_{n})$
and $h_{1}(\bullet;K)$ are as above, then also $h_{k}(\alpha;K_{n})\to h_{k}(\alpha;K)$
in the $C^{\lfloor\frac{k}{2}\rfloor}(S^{1})$ topology.\end{rem}
\begin{prop}
\label{prop:disjoint_singular_supports}Let $\phi\in Val_{k}^{+}(\R^{n})^{SO(n-1)}$
be a continuous $k$-homogeneous even valuation such that $\phi(K^{n})=\int_{S^{1}}fh_{k}(\alpha;K)$
for $SO(n-1)$-invariant convex bodies $K^{n}$ with smooth $h_{k}(\bullet;K)$,
where $f\in C_{even}^{-\infty}(S^{1})$. Then $\phi(K^{n})=\int_{S^{1}}fh_{k}(\alpha;K)$
for all $SO(n-1)$-invariant symmetric convex bodies $K^{n}$ such
that $\text{sing-supp}(h_{k}(\alpha;K))$ and $\text{sing-supp}(f)$
are disjoint one from another, and $\text{sing-supp}(h_{k}(\alpha;K))$
is disjoint from the poles.\end{prop}
\begin{proof}
Denote $G=SO(k+1)$, $H=SO(k)$. Write $S^{k}=H\backslash G$ for
the space of orbits under left action. Let $d\mu$ be the Haar probability
measure on $G$, $d\sigma$ the pushforward to $S^{k}$. Fix a positive
approximate identity $F_{N}\in C^{\infty}(S^{k})^{H}$ supported near
the north pole (identified with its $H$- bi-invariant pullback to
$G)$. It can be obtained by fixing an approximate identity $\tilde{F}_{N}$
on $G$, and then taking 
\[
F_{N}(g)=\int_{H\times H}\tilde{F}_{N}(h_{1}gh_{2})dh_{1}dh_{2}
\]
Note that $F_{N}(g)=F_{N}(g^{-1})$ by bi-invariance of $F_{N}$,
and since $\langle gH,H\rangle=\langle H,g^{-1}H\rangle$ (considered
as points on the sphere).\\
Convolution of functions is defined by 
\[
u\ast v(x)=\int_{G}u(g)v(g^{-1}x)d\mu(g)=\int_{G}v(g)u(xg^{-1})d\mu(g)
\]
so that $(L_{h}u)\ast v=L_{h}(u\ast v)$ and $R_{h}(u\ast v)=u\ast R_{h}v$
(here $L_{h}$and $R_{h}$ denote the left and right actions respectively).
In particular, for $u\in C^{\infty}(S^{k})$, $v\in C^{\infty}(G)$,
$u\ast v\in C^{\infty}(S^{k})$, and if $v$ is right $H$-invariant,
so is $u\ast v$. The following properties hold: \label{Convolution properties}\\
1. Convolution with $F_{N}$ on either side is self adjoint: for $u,v\in C^{\infty}(S^{k})$,
$\langle F_{N}\ast u,v\rangle=\langle u,F_{N}\ast v\rangle$and $\langle u\ast F_{N},v\rangle=\langle u,v\ast F_{N}\rangle$.
For instance, 
\[
\langle F_{N}\ast u,v\rangle=\int_{G}d\mu(x)v(Hx)\int_{G}d\mu(g)u(Hg)F_{N}(xg^{-1})=\int_{G\times G}d\mu(x)d\mu(g)v(Hx)u(Hg)F_{N}(xg^{-1})
\]
and we can exchange $x$ and $g$ since $F_{N}(xg^{-1})=F_{N}(gx^{-1})$.
Similarly,
\[
\langle u\ast F_{N},v\rangle=\int_{G}d\mu(x)v(Hx)\int_{G}d\mu(g)u(Hg)F_{N}(g^{-1}x)=\int_{G\times G}d\mu(x)d\mu(g)v(Hx)u(Hg)F_{N}(g^{-1}x)
\]
2. For $u\in C^{\infty}(S^{k})$, one has $F_{N}\ast u\to u$ in $C^{\infty}(S^{k})$.
For $u\in C^{\infty}(G/H)$, $u\ast F_{N}\to u$. 
\[
F_{N}\ast u(x)=\int_{H\times H}dh_{1}dh_{2}\int_{G}\tilde{F}_{N}(h_{1}gh_{2})u(g^{-1}x)dg=\int_{H\times H}dh_{1}dh_{2}\int_{G}\tilde{F}_{N}(h_{1}g)u(h_{2}^{-1}g^{-1}x)dg
\]
by left $H$-invariance of $u$, this equals 
\[
\int_{H}dh\int_{G}\tilde{F}_{N}(hg)u(g^{-1}x)dg=\int_{H}dh\tilde{F}_{N}\ast u(hx)dh=\int_{H}L_{h}(\tilde{F}_{N}\ast u)(x)dh
\]
since $L_{h}(\tilde{F}_{N}\ast u)(x)\to L_{h}u(x)=u(x)$ in $C^{\infty}(G)$,
we conclude that $\int_{H}L_{h}(\tilde{F}_{N}\ast u)(x)dh\to u$ in
$C^{\infty}(S^{k})$. Similarly, for $u\in C^{\infty}(G/H)$, 
\[
u\ast F_{N}(x)=\int_{H\times H}dh_{1}dh_{2}\int_{G}\tilde{F}_{N}(h_{1}gh_{2})u(xg^{-1})dg=\int_{H}dh_{2}\int_{G}\tilde{F}_{N}(gh_{2})u(xg^{-1})dg=
\]
\[
=\int_{H}dh\int_{G}\tilde{F}_{N}(g)u(xh^{-1}g^{-1})dg=\int_{H}R_{h}(u\ast\tilde{F}_{N})(x)dh
\]
and again 
\[
R_{h}(u\ast\tilde{F}_{N})\to R_{h}u=u
\]
implying the statement.\\
3. For $u\in C^{-\infty}(S^{k})$, $F_{N}\ast u\to u$ for $u\in C^{-\infty}(S^{k})$
and $u\ast F_{N}\to u$ for $u\in C^{-\infty}(G/H)$. This is a direct
consequence of properties 1 and 2.\\
4. For $u\in C^{-\infty}(S^{k})$, $T_{k}(u\ast F_{N})=T_{k}(u)\ast F_{N}$.
It is enough by self-adjointness of $T_{k}$ and the convolution operator
to verify this for $u\in C^{\infty}(S^{k})$: 
\[
T_{k}(u\ast F_{N})(x)=\int_{S^{k}}dy|\langle x,y\rangle|\int_{G}dgF_{N}(g)u(yg^{-1})=
\]
\[
=\int_{G}dgF_{N}(g)\int_{S^{k}}u(yg^{-1})|\langle x,y\rangle|dy=\int_{G}dgF_{N}(g)\int_{S^{k}}u(y)|\langle xg^{-1},y\rangle|dy=
\]
\[
=\int_{G}dgF_{N}(g)T_{k}u(xg^{-1})=T_{k}u\ast F_{N}(x)
\]
Note that $F_{N}\ast u\in C^{-\infty}(S^{k})^{H}$ whenever $u\in C^{-\infty}(S^{k})^{H}$.
\\
\\
Let $\sigma_{k}\in C^{-\infty}(S^{k})^{H}$ be the surface area measure
of $K^{k+1}$. Then by Minkowski's theorem, $\sigma_{k}\ast F_{N}$
is the surface area measure of a sequence of $H$-invariant bodies
denoted $K_{N}^{k+1}$ s.t. $K_{N}\to K$, therefore also $K_{N}^{n}\to K^{n}$
and $\phi(K_{N}^{n})\to\phi(K^{n})$. On the other hand,
\[
T(\sigma_{k}\ast F_{N})=h_{k}(\bullet;K)\ast F_{N}
\]
so 
\[
\phi(K_{N}^{n})=\int_{S^{1}}f\cdot T(\sigma_{k}\ast F_{N})d\alpha=\int_{S^{1}}f(\alpha)\cdot(h_{k}(\bullet;K)\ast F_{N})(\alpha)d\alpha
\]
Choose a cut-off function $\chi\in C^{\infty}(S^{1})^{\mathbb{Z}_{2}}$
(the action is reflection w.r.t. the vertical axis, note that $\chi$
induces a smooth $H-$invariant function on $S^{k}$, also denoted
$\chi$) such that $\chi(\alpha)h_{k}(\bullet;K)\in C^{\infty}(S^{k})$
and $(1-\chi(\alpha))f(\alpha)\in C^{\infty}(S^{1})$, and $\chi=1$
in a neighborhood of the poles. Now we can restrict $\chi(\alpha)h_{k}(\bullet;K)$
to a smooth function on $S^{1}$, and 
\[
\chi(\alpha)(F_{N}\ast h_{k}(\bullet;K))(\alpha)\to\chi(\alpha)h_{k}(\alpha;K)
\]
in $C^{\infty}(S^{k})$ and also in $C^{\infty}(S^{1})$ (by restriction).
Then 
\[
\int_{S^{1}}f(\alpha)(h_{k}(\bullet;K)\ast F_{N})(\alpha)d\alpha=\int_{S^{1}}f(\alpha)\Big(\chi(\alpha)(h_{k}(\bullet;K)\ast F_{N})(\alpha)\Big)d\alpha+
\]
\[
+\int_{S^{1}}\Big((1-\chi(\alpha))f(\alpha)\Big)(h_{k}(\bullet;K)\ast F_{N})(\alpha)d\alpha
\]
The first summand converges to 
\[
\int_{S^{1}}f(\alpha)\chi(\alpha)h_{k}(\alpha;K)d\alpha
\]
Also, $(1-\chi(\alpha))f(\alpha)$ can be pulled back to a smooth
function on $S^{k}$ since $1-\chi=0$ near the poles. In particular,
we will have 
\[
\int_{S^{1}}\Big((1-\chi(\alpha))f(\alpha)\Big)(h_{k}(\bullet;K)\ast F_{N})(\alpha)d\alpha\to\int_{S^{1}}\Big((1-\chi(\alpha))f(\alpha)\Big)h_{k}(\alpha;K)(\alpha)d\alpha
\]
And so the sum converges to $\int_{S^{1}}f(\alpha)h_{k}(\alpha;K)d\alpha$,
as required.
\end{proof}

\section{\label{sec:nonexistence of even Lorentz}Finding the generalized
invariant valuations}

From now on, $n\geq3$, and $G=SO^{+}(n-1,1)$. Let us recall some
definitions and facts and introduce notation.\\
\\
Consider the bundle $E^{n,k}$ over $Gr(V,n-k)$ with fiber over $\Lambda\in Gr(V,n-k)$
equal to $E^{n,k}|_{\Lambda}=D(V/\Lambda)\otimes D(T_{\Lambda}Gr(V,n-k))$.
We will sometimes refer to it as the Crofton bundle, and we call its
(generalized) sections (generalized) Crofton measures. Also, recall
Klain's bundle $K^{n,k}$ over $Gr(V,k)$, that has fiber $D(\Lambda)$
over $\Lambda\in Gr(n,k)$. Klain's imbedding $Kl:Val_{k}^{ev}(V)\to\Gamma(K^{n,k})$
is $GL(V)$-equivariant, and maps smooth valuations to smooth sections,
see \cite{alesker-adv-2000}.\\
\\
Observe there is a natural bilinear non-degenerate pairing 
\[
\Gamma^{\pm\infty}(E^{n,k})\times\Gamma^{\mp\infty}(K^{n,n-k})\to D(V)
\]
The $GL(V)-$equivariant cosine transform $T_{n-k,k}:\Gamma^{\infty}(E^{n,k})\to\Gamma^{\infty}(K^{n,k})$
is given by 
\[
T_{n-k,k}(\gamma)(D_{\Lambda})=\int_{\Omega\in Gr(n,n-k)}\gamma\otimes Pr_{V/\Omega}(D_{\Lambda})
\]
where $D_{\Lambda}\subset\Lambda$ is some symmetric convex body.
We will write $T_{n-k,k}:C^{\infty}(Gr(n,n-k))\to C^{\infty}(Gr(n,k))$
also for the cosine transform after a Euclidean trivialization, and
also $T_{n-k,k}:\Gamma^{-\infty}(E^{n,k})\to\Gamma^{-\infty}(K^{n,k})$
for the adjoint operator to $T_{k,n-k}:\Gamma^{\infty}(E^{n,n-k})\to\Gamma^{\infty}(K^{n,n-k})$.
It extends the cosine transform on smooth sections.

\subsection{Some representation theory}

We make use of the following facts (see \cite{AleskerBernstein}):
\begin{enumerate}
\item The highest weights of $SO(n)$ are parametrized by sequences of integers
$\lambda=(\lambda_{1},...,\lambda_{\lfloor\frac{n}{2}\rfloor})$ with
$\lambda_{1}\geq...\geq\lambda_{\lfloor\frac{n}{2}\rfloor}\geq0$
for odd $n$, and $\lambda_{1}\geq...\geq\lambda_{\lfloor\frac{n}{2}\rfloor-1}\geq|\lambda_{\lfloor\frac{n}{2}\rfloor}|$
for even $n>2$.
\item The irreducible components of $C^{\infty}(Gr(n,k))$ (considered as
a representation of $SO(n)$) are of multiplicity one, with highest
weights $\lambda\in\Lambda_{k}^{+}\cap\Lambda_{n-k}^{+}$. Here $\Lambda_{j}=\{\lambda:\lambda_{i}=0\,\forall i>j,\,\lambda_{i}\equiv0\mod2\,\forall i\}$.
\item The image of $T_{k}:C^{\infty}(Gr(n,n-k))\to C^{\infty}(Gr(n,k))$
consists of representations with highest weights $\lambda\in\Lambda_{k}^{+}\cap\Lambda_{n-k}^{+}$,
$|\lambda_{2}|\leq2$. The kernel is thus $\text{Ker}T_{k}=\oplus\rho_{\lambda}$
with $\lambda\in\Lambda_{k}^{+}\cap\Lambda_{n-k}^{+}$, $|\lambda_{2}|\geq4$.
The image of $T_{k}$ is closed.
\item The irreducible representations of $SO(n)$ which contain an $SO(n-1)$-invariant
element are precisely those corresponding to spherical harmonics.
Their highest weight is $(d,0,...,0)$ (for degree $d$ spherical
harmonics). The spherical harmonics appearing in $C^{\infty}(G(n,n-k))$
are precisely those of even degree $d$. 
\item In particular, $C^{\infty}(Gr(n,n-k))^{SO(n-1)}\cap\text{Ker}T_{n-k,k}=0$.
Thus 
\begin{equation}
T_{n-k,k}:C^{\infty}(Gr(V,n-k))^{SO(n-1)}\to C^{\infty}(Gr(V,k))^{SO(n-1)}\label{eq:cosine transform is isomorphism}
\end{equation}
is an isomorphism: It is injective and has dense image (by Schur's
Lemma), and also 
\begin{equation}
T_{n-k,k}\Big((C^{\infty}(Gr(V,n-k)))^{SO(n-1)}\Big)=\Big(T_{n-k,k}(C^{\infty}(Gr(V,n-k)))\Big)^{SO(n-1)}\label{eq:cosine_image}
\end{equation}
implying the image is closed. Equation \ref{eq:cosine_image} holds
because $T_{n-k,k}$ obviously maps $SO(n-1)$-invariant vectors to
$SO(n-1)$-invariant vectors, and if $v\in T_{n-k,k}(C^{\infty}(Gr(V,n-k)))$
is $SO(n-1)$-invariant, then $v=T_{n-k,k}u$ for some $u\in C^{\infty}(Gr(V,n-k))$
such that $v=T_{n-k,k}(gu)$ for all $g\in SO(n-1)$, implying $v=T_{n-k,k}(\int_{SO(n-1)}gu\cdot dg)$.
\item In particular, 
\[
T_{n-k,k}:C^{-\infty}(Gr(V,n-k))^{SO(n-1)}\to C^{-\infty}(Gr(V,k))^{SO(n-1)}
\]
is also an isomorphism, since $T_{n-k,k}$ is a symmetric opeator
(after the obvious identification $Gr(V,k)=Gr(V,n-k)$).
\end{enumerate}
Note that the action of $SO(n-1)$ on $\Gamma^{-\infty}(E^{n,k})$
and $\Gamma^{-\infty}(K^{n,k})$ (after a Euclidean trivialization)
and on $C^{-\infty}(Gr(V,n-k))$ resp. $C^{-\infty}(Gr(V,k))$ coincides.
We deduce the following 
\begin{cor}
\label{cor:The-map-is-an-injection}The map 
\[
T_{n-k,k}:\Gamma^{-\infty}(E^{n,k})^{SO^{+}(n-1,1)}\to\Gamma^{-\infty}(K^{n,k})^{SO^{+}(n-1,1)}
\]
is injective.
\end{cor}
Let us prove the following 
\begin{prop}
\label{prop:smooth image is smooth functions intersect generalized image}$C^{\infty}(Gr(n,n-k))\cap T_{n-k,k}(C^{-\infty}(Gr(n,n-k)))=C^{\infty}(Gr(n,n-k))$\end{prop}
\begin{proof}
Assume $h(\Lambda)=T_{k,n-k}(\sigma)$ for some $\sigma\in C^{-\infty}(Gr(n,k))$
and $h\in C^{\infty}(Gr(n,n-k))$. Choose an approximate identity
$\mu_{N}\in\M^{\infty}(SO(n))$. Then $T_{k,n-k}(\sigma\ast\mu_{N})=T_{k,n-k}(\sigma)\ast\mu_{N}=h\ast\mu_{N}\to h$
in the $C^{\infty}$-topology. Since $\sigma\ast\mu_{N}\in C^{\infty}(Gr(n,k))$,
and the image of $T_{k,n-k}$ is closed in the $C^{\infty}$ topology,
it follows that $h\in T_{n-k,k}(C^{\infty}(Gr(n,k)))$, as claimed. 
\end{proof}

\subsection{\label{sub:Lorentz-invariant-generalized-valuations}Lorentz-invariant
generalized valuations}

The space $Val_{k}^{ev,-\infty}(V)$ of generalized $k$-homogeneous
even valuations is defined by 
\[
Val_{k}^{ev,-\infty}(V)=\Big(Val_{n-k}^{ev,\infty}(V)\Big)^{*}\otimes D(V)=\Big(Val_{n-k}^{ev,\infty}(V)\otimes D(V)^{*}\Big)^{*}
\]
By the Alesker-Poincare duality, there is a natural inclusion $Val_{k}^{ev,\infty}(V)\subset Val_{k}^{ev,-\infty}(V)$.
\\
Let us write this inclusion explicitly. Recall that a Crofton measure
$\mu_{\phi}\in\Gamma^{\infty}(Gr(V,n-k),E^{n,k})$ for $\phi\in Val_{k}^{ev,\infty}(V)$
is any section such that $T_{n-k,k}(\mu_{\phi})=Kl(\phi)$, which
always exists by \cite{AleskerBernstein}. It is equivalent to a smooth,
translation-invariant measure on the affine Grassmannian $\overline{Gr}(V,n-k)$.\\
\\
For $\phi\in Val_{k}^{ev,\infty}(\R^{n})$ and $\psi\in Val_{n-k}^{ev,\infty}(\R^{n})$,
the duality map is given by
\[
\langle\phi,\psi\rangle=\langle Kl(\phi),\mu_{\psi}\rangle
\]
Equivalently,
\[
\langle\phi,\psi\rangle(\bullet)=\int_{\overline{Gr}(V,k)}\phi(\bullet\cap E)d\mu_{\psi}(E)\in D(V)
\]
We have the surjective map 
\[
Cr_{k}:\Gamma^{\infty}(E^{n,n-k})\to Val_{n-k}^{ev,\infty}(V)
\]
given by 
\[
Cr_{k}(s)(K)=\int_{\Lambda\in Gr(V,k)}s(Pr_{V/\Lambda}(K))
\]

We will need the following 
\begin{claim}
\label{Claim: closed image dual}Let $T:X\to Y$ be a bounded linear
map between Frechet spaces $X,Y$ such that $Im(T)\subset Y$ is closed.
Then $Im(T^{*})\subset X^{*}$ is also closed.\end{claim}
\begin{proof}
By Banach's open mapping theorem, $T:X/\text{Ker}(T)\to\text{Im}(T)$
is an isomorphism of Frechet spaces. Therefore, $T^{*}:\text{Im}(T)^{*}\to(X/\text{Ker}(T))^{*}=\text{Ker}(T)^{\perp}$
is also an isomorphism. It remains to observe that $T^{*}:Y^{*}\to X^{*}$
factorizes as $Y^{*}\twoheadrightarrow\text{Im}(T)^{*}\simeq\text{Ker}(T)^{\perp}\hookrightarrow X^{*}$
and the last inclusion is closed. \end{proof}
\begin{prop}
\label{prop:KlainExtension}There is a unique extension by continuity
of Klain's imbedding, $Kl_{k}:Val_{k}^{ev,-\infty}(V)\to\Gamma^{-\infty}(K^{n,k})$,
which is an imbedding with closed image.
\end{prop}
Consider the adjoint map of $Cr_{k}$: 
\[
Cr_{k}^{*}:Val_{k}^{ev,-\infty}(V)\otimes D(V)^{*}\to\Gamma^{-\infty}(K^{n,k})\otimes D(V)^{*}
\]
which gives a map 
\[
A:Val_{k}^{ev,-\infty}(V)\to\Gamma^{-\infty}(K^{n,k})
\]
s.t. $Cr_{k}^{*}=A\otimes Id$. Let us verify that $A$ extends Klain's
imbedding $Kl_{k}:Val_{k}^{ev,\infty}(V)\to\Gamma^{\infty}(K^{n,k})$.
For $\gamma\in\Gamma^{\infty}(E^{n,n-k})$, one has the obvious Crofton
measure $\mu_{Cr_{k}(\gamma)}=\gamma$, so for all $\psi\in Val_{k}^{ev,\infty}(V)$
\[
A(\psi)(\gamma)=\langle Cr_{k}(\gamma),\psi\rangle=\int_{Gr(V,k)}\mu_{Cr_{k}(\gamma)}Kl_{k}(\psi)=
\]
\[
=\int_{Gr(V,k)}\gamma Kl_{k}(\psi)=\langle\gamma,Kl_{k}(\psi)\rangle
\]
as required. Moreover, $\text{Ker}A=0$, since $Cr_{k}$ is surjective,
and by Claim \ref{Claim: closed image dual} the image of $A$ is
closed.
\begin{prop}
\label{prop:CroftonExtension}The map $Cr_{k}$ admits a unique extension
by continuity $Cr_{k}:\Gamma^{-\infty}(E^{n,n-k})\to Val_{n-k}^{ev,-\infty}(V)$
which is surjective. It holds that $Kl_{n-k}\circ Cr_{k}=T_{k,n-k}$.
\end{prop}
Consider the dual to Klain's imbedding $Kl_{k}:Val_{k}^{ev,\infty}(V)\to\Gamma^{\infty}(K^{n,k})$,
tensored with the identity on $D(V)$: It is given by 
\[
B:\Gamma^{-\infty}(E^{n,n-k})\to Val_{n-k}^{ev,-\infty}(V)
\]
where 
\[
B(s)(\psi)=\langle s,Kl_{k}(\psi)\rangle
\]
for all $\psi\in Val_{k}^{ev,\infty}(V)$. Then $B$ extends the Crofton
surjection: for $\gamma\in\Gamma^{\infty}(K^{n,n-k})$ and $\psi\in Val_{k}^{ev,\infty}(V)$,
\[
B(\gamma)(\psi)=\langle\gamma,Kl_{k}(\psi)\rangle=\langle Cr_{k}(\gamma),\psi\rangle
\]
Let us verify it is surjective: the image of $B$ is dense since $Kl_{k}$
is injective. The image of $B$ is closed by Claim \ref{Claim: closed image dual}
since $\text{Im}(Kl_{k})$ is closed. Note that 
\[
Cr_{n-k}^{*}\circ Kl_{k}^{*}=(Kl_{k}\circ Cr_{n-k})^{*}=T_{n-k,k}^{*}=T_{k,n-k}
\]
implying $B\circ Cr_{k}=T_{k,n-k}$. 
\begin{defn}
A generalized Crofton measure for $\phi\in Val_{k}^{ev,-\infty}(V)$
is any $\mu\in\Gamma^{-\infty}(E^{n,n-k})$ s.t. $Cr_{k}(\mu)=\phi.$
We proved that such $\phi$ exists.
\end{defn}

\subsection{Reconstructing a continuous valuation from its generalized Crofton
measure}
\begin{lem}
\label{croftonMeasureDefinesValuationOnSmoothBodies}Let $W$ be a
linear space, $\phi\in Val_{k}^{ev}(W)$ a continuous valuation, and
$\mu_{\phi}\in\Gamma^{-\infty}(E^{n,k})$ a generalized Crofton measure
for $\phi$. Let $K$ be a convex body such that $|Pr_{W/\Lambda}(K)|\in\Gamma^{\infty}(K^{n,n-k})\otimes D(W)^{*}$.
Then 
\[
\phi(K)=\int_{Gr(n,n-k)}|Pr_{W/\Lambda}(K)|\mu_{\phi}(\Lambda)
\]
\end{lem}
\begin{proof}
A convex body $K\subset W$ is naturally an element of $Val_{k}^{ev,\infty}(W)^{*}=Val_{n-k}^{ev,-\infty}(W)\otimes D(W)^{*}$;
denote the corresponding element by $\psi{}_{K,n-k}$. Then $\psi_{K,n-k}=Kl^{*}(\gamma_{K,n-k})=(Cr\otimes Id)(\gamma_{K,n-k})$
for some $\gamma_{K,n-k}\in\Gamma^{-\infty}(E^{n,n-k})\otimes D(W)^{*}$,
and so 
\[
Cr^{*}(\psi_{K,n-k})=(Kl_{n-k}\otimes Id)(\psi_{K,n-k})=(T_{k,n-k}\otimes Id)(\gamma_{K,n-k})\in\Gamma^{-\infty}(K^{n,n-k})\otimes D(W)^{*}
\]
In particular, $Cr^{*}(\psi_{K,n-k})$ lies in the image of the cosine
transform. \\
Let us verify that $Cr^{*}(\psi_{K,n-k})$ is continuous and $Cr^{*}(\psi_{K,n-k})(\Lambda)=|Pr_{W/\Lambda}(K)|\in\Gamma(K^{n,n-k})\otimes D(W)^{*}$,
where $\Lambda\in Gr(V,n-k)$. \\
Take any smooth Crofton measure $\gamma\in\Gamma^{\infty}(E^{n,k})$.
Then
\[
\langle Cr^{*}(\psi_{K,n-k}),\gamma\rangle=\langle\psi_{K,n-k},Cr(\gamma)\rangle=Cr(\gamma)(K)=\int_{Gr(n,n-k)}|Pr_{W/\Lambda}(K)|\gamma
\]
that is, $Cr^{*}(\psi_{K,n-k})=|Pr_{W/\Lambda}(K)|$, so $|Pr_{W/\Lambda}(K)|\in T_{k,n-k}(\Gamma^{-\infty}(E^{n,n-k}))\otimes D(W)^{*}$.
By Proposition \ref{prop:smooth image is smooth functions intersect generalized image},
it follows that $|Pr_{W/\Lambda}(K)|=T_{k,n-k}(\sigma)$ for some
$\sigma\in\Gamma^{\infty}(E^{n,n-k})\otimes D(W)^{*}$.\\
\\
Now fix some Euclidean structure on $W$. We know that $T_{n-k,k}(\mu_{\phi})=Kl(\phi)$.
Choose a sequence $\phi_{j}\in Val_{k}^{ev,\infty}(W)$ s.t. $\phi_{j}\to\phi$
in $Val_{k}^{ev}(W)$, so $\phi_{j}(K)\to\phi(K)$. Choose Crofton
measures $\mu_{n}\in\Gamma^{\infty}(E^{n,k})$ s.t. $T_{k,n-k}(\mu_{j})=Kl(\phi_{j})$.
Then since $T_{k,n-k}^{*}=T_{n-k,k}$, 
\[
\phi_{j}(K)=\int_{Gr(n,n-k)}|Pr_{\Lambda^{\perp}}(K)|\mu_{j}(\Lambda)=\int_{Gr(n,k)}\sigma T_{n-k,k}(\mu_{j})=
\]
\[
=\int_{Gr(n,k)}\sigma Kl(\phi_{j})\to\int_{Gr(n,k)}\sigma Kl(\phi)=\int_{Gr(n,k)}\sigma T_{n-k,k}(\mu_{\phi})
\]
and since $\sigma$ is smooth and $T_{k,n-k}^{*}=T_{n-k,k}$, this
equals
\[
\int_{Gr(n,n-k)}|Pr_{\Lambda^{\perp}}(K)|\mu_{\phi}(\Lambda)
\]
as claimed.
\end{proof}
Thus, given a generalized section $s\in\Gamma^{-\infty}(E^{n,k})^{SO^{+}(n-1,1)}$,
we may consider $\phi=Cr_{n-k}(s)$ which is an even, $k$-homogeneous,
Lorentz-invariant generalized valuation. Then one may ask whether
a continuous extension to all convex bodies of $\phi$ exists. According
to the Lemma, its value (as a continuous valuations) on all convex
bodies with smooth $k$-support function should be given by the formula
\[
\phi(K)=\int_{\Lambda\in Gr(V,n-k)}s(Pr_{V/\Lambda}(K))
\]

\subsection{Finding the invariant generalized sections}

Let $X$ be a smooth manifold, and $Y\subset X$ a smooth compact
submanifold. Let $E$ be a smooth vector bundle over $Y$. Define
the sheaf 
\[
J_{Y}^{q}=\{f\in C^{\infty}(X):\,\, L_{X_{1}}...L_{X_{q}}f\Big|_{Y}=0\,\,\forall X_{j}\in\Gamma^{\infty}(TX),\,\, j=1,...,q\}
\]
Then define $\M_{Y}^{q}=J_{Y}^{q}\M^{\infty}(X)$ and 
\[
\Gamma_{Y}^{-\infty,q}(E)=\{\phi\in\Gamma^{-\infty}(E):\forall s\in\Gamma^{\infty}(E^{*}),m\in\Gamma(\M_{Y}^{q})\,\,\phi(s\otimes m)=0\}
\]
Let $F^{q}$ denote the vector bundle over $Y$ with fiber 
\[
F^{q}|_{x}=Sym^{q}(N_{x}Y)\otimes D^{*}(N_{x}Y)\otimes E|_{x}
\]
where $N_{x}Y=T_{x}X/T_{x}Y$ is the normal space to $Y$ at $x$.
Then $\Gamma_{Y}^{-\infty,q}(E)/\Gamma_{Y}^{-\infty,q-1}(E)=\Gamma^{-\infty}(Y,F^{q})$.
\\
\\
We thus have a useful tool for finding the $G$-invariant generalized
sections of a vector bundle:
\begin{prop}
\label{prop:SectionsWithSubmanifoldSupport}Let $G$ be a group, $X$
a manifold equipped with $G$-action, $E$ over $X$ a $G$-equivariant
line bundle, and $Y\subset X$ a compact orbit of $G$. Then there
is an injective map 
\[
p:\Big(\Gamma_{Y}^{-\infty,q}(E)/\Gamma_{Y}^{-\infty,q-1}(E)\Big)^{G}\to\Gamma^{\infty}(Y,F^{q})^{G}
\]
\end{prop}
\begin{proof}
Taking the $G$-invariant elements of a $G$-module is left exact.
Therefore, the exact sequence 
\[
0\to\Gamma_{Y}^{-\infty,q-1}(E)\to\Gamma_{Y}^{-\infty,q}(E)\to\Gamma^{-\infty}(Y,F^{q})\to0
\]
gives an injection 
\[
\Big(\Gamma_{Y}^{-\infty,q}(E)/\Gamma_{Y}^{-\infty,q-1}(E)\Big)^{G}\to\Gamma^{-\infty}(Y,F^{q})^{G}
\]
So it remains to verify that in fact $\Gamma^{-\infty}(Y,F^{q})^{G}\subset\Gamma^{\infty}(Y,F^{q})$.
This holds because $G$ acts transitively on $Y$: we can choose any
smooth probability measure with compact support $\mu$on $G$ , and
then $\forall f\in\Gamma^{-\infty}(Y,F^{q})^{G}$, $f=f\ast\mu\in\Gamma^{\infty}(Y,F^{q})$.
\end{proof}

\subsubsection{\label{sec:Some-generalized-functions}Construction of some generalized
functions on the unit circle}

For the following, define $c_{j}(\lambda)$ by 
\[
\Big(\frac{\sin x}{x}\Big)^{\lambda}=\sum_{j=0}^{\infty}c_{j}(\lambda)x^{2j}
\]
The series converge locally uniformly in $x\in(-\pi,\pi)$ for every
$\lambda\in\C$, in particular $\sum_{j=0}^{\infty}|c_{j}(\lambda)|$
converges. The coefficients $c_{j}(\lambda)$ are polynomial functions
of $\lambda\in\C$: $c_{0}(\lambda)=1$, $c_{1}(\lambda)=-\frac{\lambda}{3!}$,
$c_{2}(\lambda)=\frac{\lambda}{5!}+\frac{\lambda(\lambda-1)}{2\cdot3!^{2}}$,
$c_{3}(\lambda)=-\frac{\lambda}{7!}-\frac{\lambda(\lambda-1)}{3!5!}+\frac{\lambda(\lambda-1)(\lambda-2)}{6\cdot3!^{3}}$
and so on. 
\begin{lem}
For every $k\in\Z$, the function $I_{k}(\lambda):\C\to\C$ given
by 
\[
I_{k}(\lambda)=\int_{0}^{1}x^{k}|\sin x|^{\lambda}dx
\]
for $\text{Re}\lambda>0$, admits a meromorphic extension to the complex
plane, with simple poles at $\lambda=-(k+2j+1)$, $j=0,1,2,...$ and
residues $Res(I_{k},-k-2j-1)=2c_{j}(-k-2j-1)$.\end{lem}
\begin{proof}
Write

\[
I_{k}(\lambda)=\int_{0}^{1}x^{k+\lambda}\Big(\frac{\sin x}{x}\Big)^{\lambda}dx=
\]
\[
=\sum_{j=0}^{\infty}c_{j}(\lambda)\frac{1}{\lambda+k+2j+1}
\]
is meromorphic with simple poles at $\lambda=-k-2j-1$, $j\geq0$.\end{proof}
\begin{lem}
There exists a meromorphic map $\sin_{+}^{\lambda}x:\C\to C^{-\infty}(-\pi,\pi)$
with simple poles at $\lambda=-1,-2,...$ and residues 
\[
\text{Res}(\sin_{+}^{\lambda},-k)=\Bigg\{\begin{array}{c}
\sum_{j=0}^{m}\frac{1}{(2j)!}c_{m-j}(-k)\delta_{0}^{(2j)},k=2m+1\\
-\sum_{j=0}^{m-1}\frac{1}{(2j+1)!}c_{m-1-j}(-k)\delta_{0}^{(2j+1)},k=2m
\end{array}
\]
 s.t. for all $\lambda\notin\Z_{<0}$, $\sin_{+}^{\lambda}x(\phi dx)=\int_{0}^{\pi}\phi(x)\sin^{\lambda}xdx$
for $\phi\in C_{c}^{\infty}(-\pi,\pi)$ that vanishes in a neighborhood
of 0. \end{lem}
\begin{proof}
For $Re(\lambda)>-1$, $\sin_{+}^{\lambda}x$ is locally integrable
near 0 and so $\sin_{+}^{\lambda}x\in C^{-\infty}(-1,1)$ is well-defined
and analytic in $\lambda$. A meromorphic continuation with the desired
properties in the region $Re(\lambda)>-(k+1)$ is given for $\phi\in C_{c}(-\pi,\pi)$
by 
\[
\sin_{+}^{\lambda}x(\phi dx)=\int_{1}^{\pi}\phi(x)\sin^{\lambda}xdx+\int_{0}^{1}\sin^{\lambda}x(\phi(x)-\phi(0)-x\phi'(0)-...-\frac{1}{(k-1)!}x^{k-1}\phi^{(k-1)}(0))dx+
\]
\[
+\phi(0)I_{0}(\lambda)+\phi'(0)I_{1}(\lambda)+...+\frac{1}{(k-1)!}\phi^{(k-1)}(0)I_{k-1}(\lambda)
\]
by the Lemma above, this is a well-defined generalized function, meromorphic
in $\lambda$, with simple poles at $\lambda=-1,-2,...$ and residues
as claimed.
\end{proof}
We define also $\sin_{-}^{\lambda}x\in C^{-\infty}(-\pi,\pi)$ by
$\langle\sin_{-}^{\lambda}x,\phi(x)dx\rangle=\langle\sin_{+}^{\lambda}x,\phi(-x)dx\rangle$.
Then 
\[
\text{Res}(\sin_{-}^{\lambda}x,-k)=\Bigg\{\begin{array}{c}
\sum_{j=0}^{m}\frac{1}{(2j)!}c_{m-j}(-k)\delta^{(2j)},k=2m+1\\
\sum_{j=0}^{m-1}\frac{1}{(2j+1)!}c_{m-1-j}(-k)\delta^{(2j+1)},k=2m
\end{array}
\]
Before formulating the main result of this subsection, recall the
following
\begin{claim*}
Let $f:\C\to C^{-\infty}(X)$, $\lambda\mapsto f_{\lambda}(x)$ be
meromorphic, where $X$ is a smooth manifold. Assume that $\lambda_{0}$
is a simple pole, and $h(x)\in C(X)$ positive s.t. $f_{\lambda}(gx)=h(x)^{\lambda}f_{\lambda}(x)$
in the holomorphic domain of $f_{\lambda}$, for some $g\in\text{Diff}(X)$.
Then $r(x)=\text{Res}(f_{\lambda};\lambda_{0})$ satisfies the same
equation.\end{claim*}
\begin{proof}
Indeed, write $f_{\lambda}(x)=\frac{a_{-1}(x)}{\lambda-\lambda_{0}}+a_{0}(x)+...$
, so that $r(x)=a_{-1}(x)$. Then 
\[
f_{\lambda}(gx)=h(x)^{\lambda}f_{\lambda}(x)\Rightarrow\frac{a_{-1}(gx)}{\lambda-\lambda_{0}}+a_{0}(gx)+...=\frac{a_{-1}(x)h(x)^{\lambda}}{\lambda-\lambda_{0}}+a_{0}(x)h(x)^{\lambda}+...
\]
Developing $h(x)^{\lambda}$ into power series near $\lambda=\lambda_{0}$
we see that 
\[
a_{-1}(gx)=a_{-1}(x)h(x)^{\lambda_{0}}
\]
as claimed.

Recall the Lorentz form $Q$ on $\R^{2}$, which we now restrict to
the unit circle $S^{1}$. Then $\{Q\geq0\}=\{-\frac{\pi}{4}\leq\alpha\leq\frac{\pi}{4}\}\cup\{\frac{3\pi}{4}\leq\alpha\leq\frac{5\pi}{4}\}$
and $\{Q\leq0\}=\{\frac{\pi}{4}\leq\alpha\leq\frac{3\pi}{4}\}\cup\{\frac{5\pi}{4}\leq\alpha\leq\frac{7\pi}{4}\}$.\end{proof}
\begin{cor}
\label{cor:existence of equivariant sections} (a)For any sign $\epsilon\in\{+,-\}$,
there is a meromorphic in $\lambda$, generalized function $f_{\lambda}^{\epsilon}$
on $S^{1}$, namely $\cos_{\epsilon}^{\lambda}(2\alpha)$ (here $\alpha$
is the angle on the circle) with simple poles at $\lambda=-1,-2,...$
that is supported on $\sign Q\in\{0\}\cup\{\epsilon\}$, which satisfies
for every $\phi\in C^{\infty}(S^{1})$ vanishing in a neighborhood
of the light cone
\[
\langle f_{\lambda}^{\epsilon},\phi(\alpha)d\alpha\rangle=\int_{\sign Q(\alpha)=\eps}|\cos2\alpha|^{\lambda}\phi(\alpha)d\alpha
\]
and
\begin{equation}
(g^{-1})^{*}(f_{\lambda})(t)=\kappa^{\lambda}\Big(\frac{1+\kappa^{2}t^{2}}{1+t^{2}}\Big)^{-\lambda}f_{\lambda}(t)\label{eq:group_action}
\end{equation}
for $g=\left(\begin{array}{cc}
\cosh\theta & \sinh\theta\\
\sinh\theta & \cosh\theta
\end{array}\right)$, where $g_{*}(f_{\lambda})=f_{\lambda}\circ g$, $\kappa=e^{-2\theta}$,
$t=\tan(\frac{\pi}{4}-\alpha)$. \\
(b) For $\lambda=-k$, $k=1,2,...$ the residue

\[
\text{Res}(f_{\lambda}^{\epsilon};-k)=
\]
\[
=\Bigg\{\begin{array}{c}
\sum_{j=0}^{m}\frac{1}{(2j)!2^{2j}}c_{m-j}(-k)(\delta_{\alpha=\pi/4}^{(2j)}+\delta_{\alpha=5\pi/4}^{(2j)}-\delta_{\alpha=3\pi/4}^{(2j)}-\delta_{\alpha=7\pi/4}^{(2j)}),\, k=2m+1\\
-\epsilon\sum_{j=0}^{m-1}\frac{1}{(2j+1)!2^{2j+1}}c_{m-1-j}(-k)(\delta_{\alpha=\pi/4}^{(2j+1)}+\delta_{\alpha=5\pi/4}^{(2j+1)}-\delta_{\alpha=3\pi/4}^{(2j+1)}-\delta_{\alpha=7\pi/4}^{(2j+1)}),\, k=2m
\end{array}
\]
satisfies equation \ref{eq:group_action}. Also, the linear combination
\[
f_{\lambda}^{+}+(-1)^{k}f_{\lambda}^{-}
\]
is holomorphic at $\lambda=-k$ and satisfies equation \ref{eq:group_action}.\end{cor}
\begin{proof}
(a) This can be verified directly for $\text{Re}\lambda>0$, similarly
to equation \ref{eq:Klain Transformation Formula}. Then, both sides
of the equation are memoromphic maps $\C\to C^{-\infty}(S^{1})$ so
uniqueness of meromorphic extension applies. For statement (b) concerning
residues (the second half is immediate from (a)), we use the Claim
above.\end{proof}
\begin{rem}
\label{Remark Cone Symmetric And Antisymmetric}All the generalized
functions on $S^{1}$ that we defined are even, and so define generalized
functions on $\R\P^{1}$. Let $Q$ denote the Lorentz quadratic form
on $\R^{2}$. The $Q$-orthogonal complement of a line in $\R^{2}$
(which is the same as reflection w.r.t. to the light cone) induces
a $\Z_{2}$-action on $\R\P^{1}$ and so also on $C^{-\infty}(\R\P^{1})$.
We call $f\in C^{-\infty}(\R\P^{1})$ cone-symmetric or cone-antisymmetric
according to the action of $\Z_{2}$ on it. Then for $\lambda\neq-k$,
$\cos_{+}^{\lambda}(2\alpha)+\cos_{-}^{\lambda}(2\alpha)$ is cone-symmetric
and $\cos_{+}^{\lambda}(2\alpha)-\cos_{-}^{\lambda}(2\alpha)$ is
cone-antisymmetric; for $\lambda=-k$, there are two cases:\end{rem}
\begin{itemize}
\item $k$ is odd, then $\text{Res}(\cos_{\pm}^{\lambda}(2\alpha),-k)$
is cone-symmetric and $\cos_{+}^{\lambda}(2\alpha)-\cos_{-}^{\lambda}(2\alpha)$
is cone-antisymmetric.
\item $k$ is even, then $\text{Res}(\cos_{\pm}^{\lambda}(2\alpha),-k)$
is cone-antisymmetric and $\cos_{+}^{\lambda}(2\alpha)+\cos_{-}^{\lambda}(2\alpha)$
is cone-symmetric.
\end{itemize}
We will denote the cone-symmetric and cone-antisymmetric functions
corresponding to $\lambda$ by $f_{\lambda}^{+}(\alpha)$ and $f_{\lambda}^{-}(\alpha)$,
respectively, normalized so that $f_{-(2j+1)}^{+}=\text{Res}(f_{\lambda}^{+},-(2j+1))$
and $f_{-2j}^{-}=\text{Res}(f_{\lambda}^{-},-2j)$. Note that $f_{\lambda}^{\pm}$
is invariant to reflection w.r.t. the origin and to both coordinate
axes.\\
For non-integer $\lambda$, we write $f_{\lambda}^{T}$ and $f_{\lambda}^{S}$
for the functions corresponding to $\cos_{-}^{\lambda}(2\alpha)$
and $\cos_{+}^{\lambda}(2\alpha)$, resp. (standing for the time-like
and space-like support of the function).

\begin{rem}
\label{Remark: Order of derivative}Note that the generalized functions
supported on the light cone correspond to the residues, and they are
given by derivatives of order $k-1$ for $\lambda=-k$ since $c_{0}(\lambda)\equiv1$.
\end{rem}

We will now construct generalized functions $f_{n,k,\lambda}^{\pm}\in C^{-\infty}(Gr(n,k))$
that are $SO(n-1)$-invariant, have singular support on the light
cone, and satisfy the following transformation law under the Lorentz
group: Fix any $(k-1)$-dimensional $\tilde{\Lambda}\subset\R^{n-1}$
(the space coordinate plane), and $v\in\R^{n-1}$ orthogonal to $\tilde{\Lambda}$.
Denote $\Pi=\text{Span}\{v,e_{n}\}$. Let $g\in G$ be a $\theta$-boost
in $\Pi$, namely 
\[
g=\left(\begin{array}{cc}
\cosh\theta & \sinh\theta\\
\sinh\theta & \cosh\theta
\end{array}\right)
\]
and extended by identity in the orthogonal direction. Denote 
\[
\Lambda_{\alpha}=\tilde{\Lambda}+R_{\alpha}v
\]
where $R_{\alpha}$ denotes rotation by $\alpha$ in $\Pi$, extended
by the identity in the orthogonal directions. Then 
\begin{equation}
(g^{-1})^{*}(f_{n,k,\lambda}^{\pm})(\Lambda_{\alpha})=\kappa^{\lambda}\Big(\frac{1+\kappa^{2}t^{2}}{1+t^{2}}\Big)^{-\lambda}f_{n,k,\lambda}^{\pm}(\Lambda_{\alpha})\label{eq:group_action-grassmann}
\end{equation}
where $\langle g_{*}(f_{\lambda}),\mu\rangle=\langle f_{\lambda},(g^{-1})_{*}\mu\rangle$
, $\kappa=e^{-2\theta}$, $t=\tan(\frac{\pi}{4}-\alpha)$. \\
Here and in the following, $\alpha:Gr(n,k)\to[0,\frac{\pi}{2}]$ is
the elevation angle of $\Lambda\in Gr(n,k)$ above the space coordinate
hyperplane. \\
\\
This is achieved as follows: choose a smooth function $\chi\in C^{\infty}(S^{1})$
invariant to reflection w.r.t both coordinate axes, s.t. $\chi$ vanishes
in a $2\epsilon$-neighborhood of the poles and of the equator, and
equals $1$ outside a $3\epsilon$-neighborhood of the poles and equator.
Let $f\in C^{-\infty}(S^{1})$ be any generalized function smooth
near the poles and the equator, and invariant to reflections w.r.t.
both axes.\\
\\
Define $C_{\epsilon}=\{\Lambda\in Gr(n,k):\alpha(\Lambda)\geq\frac{\pi}{2}-\epsilon\}$
and $E_{\epsilon}=\{\Lambda\in Gr(n,k):\alpha(\Lambda)\leq\epsilon\}$.
Outside $C_{\epsilon}\cup E_{\epsilon}$, one has the well-defined
smooth submersion $\alpha:S^{n-1}\setminus(C_{\epsilon}\cup E_{\epsilon})\to(\epsilon,\frac{\pi}{2}-\epsilon)$.
So we may pull-back $\chi f$ as follows: define $u=\alpha^{*}(\chi f)\in C^{-\infty}(Gr(n,k))^{SO(n-1)}$
(which we extend to $C_{\epsilon}\cup E_{\epsilon}$ by zero). \\
Now observe that $\alpha^{2}$ is a smooth function on $E_{3\epsilon}$:
this can be seen by writing 
\[
\sin^{2}\alpha=\sum_{j=1}^{k}\langle v_{j},e_{n}\rangle^{2}
\]
where $\{v_{j}\}$ is any orthonormal basis of $\Lambda$, and $e_{n}$
the unit vector in the time direction. Also, $(\frac{\pi}{2}-\alpha)^{2}$
is smooth in $C_{3\epsilon}$. Since the function $(1-\chi)f\in C^{\infty}(S^{1})$
is smooth and invariant to reflections w.r.t. both coordinate axes,
by Lemma \ref{smooth_pullback} (applied separately near $\alpha=0$
and $\alpha=\frac{\pi}{2}$) one may define a smooth $SO(n-1)$-invariant
function $v(\Lambda)=\Big((1-\chi)f\Big)(\alpha(\Lambda))\in C^{\infty}(Gr(n,k))^{SO(n-1)}$
supported in $C_{3\epsilon}\cup E_{3\epsilon}$. Now define $Gr_{n,k}(f)=u+v\in C^{-\infty}(Gr(n,k))^{SO(n-1)}$.
\\
We now define $f_{n,k,\lambda}^{\pm}=Gr_{n,k}(f_{\lambda}^{\pm})$
for non-integer $\lambda$. Then for values of $\lambda$ satisfying
$\text{Re}\lambda>0$, verifying that $f_{n,k,\lambda}^{\pm}$ satisfies
equation \ref{eq:group_action-grassmann} amounts to a testing the
numerical equation given by \ref{eq:group_action}. As before for
$S^{1}$, we conclude by meromorphic extension that the equation is
satisfied for all values of $\lambda$ that are not odd resp. even
negative integers for $f_{n,k,\lambda}^{+}$ resp. $f_{n,k,\lambda}^{-}$.
Finally we define $f_{n,k,-2j}^{-}$ and $f_{n,k,-(2j+1)}^{+}$ by
taking the respective residues.\\
\\
Let us write an explicit formula for $f_{n,k,\lambda}^{\pm}(\mu)$
for $\mu\in\M^{\infty}(Gr(n,k))^{SO(n-1)}$. Writing $\mu=\phi(\alpha)d\Lambda$
where $d\Lambda$ is the unique $SO(n)$-invariant probability measure
on $Gr(n,k)$ we claim that 
\[
f_{n,k,\lambda}^{\pm}(\mu)=f{}_{\lambda}^{\pm}(\phi(\alpha)g_{n,k}(\alpha)d\alpha)
\]
with $g_{n,k}(\alpha)=C_{n,k}\cos^{n-k-1}\alpha\sin^{k-1}\alpha$.
\\
Indeed, by uniqueness of meromorphic continuation it is enough to
verify the formula for $\text{Re}\lambda>0$. Then $f_{\lambda}^{\pm}$
is continuous and $f_{n,k,\lambda}^{\pm}(\Lambda)=f_{\lambda}^{\pm}(\alpha(\Lambda))$.
So we may write 
\[
f_{n,k,\lambda}^{\pm}(\mu)=\int_{Gr(n,k)}f_{\lambda}^{\pm}(\alpha(\Lambda))\phi(\alpha(\Lambda))d\Lambda
\]
and integrate along submanifolds of constant elevation. It remains
to see that $\alpha_{*}(d\Lambda)=g_{n,k}(\alpha)d\alpha$. The angle
$\beta=\frac{\pi}{2}-\alpha$ betwen a random (w.r.t. the Haar measure
on $Gr(n,k)$) $k$-dimensional subspace and a fixed direction is
distributed as the angle between a random vector $x\in S^{n-1}$ (w.r.t.
the Haar measure) and a fixed $k$-subspace. Since $\{x\in S^{n-1}:\angle(v,\R^{k})=\beta\}=\{x\in S^{n-1}:x_{1}^{2}+...+x_{k}^{2}=\cos^{2}\beta\}=\Big(\cos\beta S^{k-1}\Big)\times\Big(\sin\beta S^{n-k-1}\Big)$,
\[
g_{n,k}(\alpha)=C_{n,k}\cos^{k-1}\beta\sin^{n-k-1}\beta=C_{n,k}\cos^{n-k-1}\alpha\sin^{k-1}\alpha
\]

\subsubsection{The case $k=1$}

We will denote $X=Gr(V,1)$, $M\subset X$ will be the set of $Q$-
degenerate subspaces, referred to as the light cone in $X$. We denote
by $\alpha$ the angle between a line $\Lambda\in X$ and the space
coordinate hyperplane. We start by proving
\begin{prop}
The $G$-invariant generalized sections of $K^{n,1}$ are spanned
by $|\cos2\alpha|^{\frac{1}{2}}s_{0}$ and $\text{sign}(\cos2\alpha)|\cos2\alpha|^{\frac{1}{2}}s_{0}$
, where $s_{0}$ is the Euclidean section.\end{prop}
\begin{proof}
We should only prove that there are no sections supported on the light
cone, denoted $M$. Assume $f\in\Gamma^{-\infty}(X,K^{n,1})$ is supported
on the light cone and $G$-invariant.\\
In our case, the action of $G$ on $X=Gr(n,1)$ is given by
\[
\tan\beta=\frac{\tan\alpha+\tanh\theta}{1+\tan\alpha\tanh\theta}
\]
where $g=\left(\begin{array}{cc}
\cosh\theta & \sinh\theta\\
\sinh\theta & \cosh\theta
\end{array}\right)$ and $\beta=g\alpha$. In particular 
\[
d\beta=\frac{d\alpha}{\cosh2\theta+\sin2\alpha\sinh2\theta}
\]
The action of $G$ on the fibers is given by 
\[
g_{*}(\phi s_{0})(\beta)=\phi(\alpha)\frac{|\cos2\beta|^{\frac{1}{2}}}{|\cos2\alpha|^{\frac{1}{2}}}s_{0}(\beta)
\]
(with the value at $\alpha=\beta=\frac{\pi}{4}$ understood in the
limit sense). We change the coordinates as follows: $\epsilon=\frac{\pi}{4}-\alpha$,
$\eta=\frac{\pi}{4}-\beta$ and $t=\tan\epsilon$, $s=\tan\eta$.
Also, denote 
\[
\kappa=\frac{1-\tanh\theta}{1+\tanh\theta}=\frac{1}{(\cosh\theta+\sinh\theta)^{2}}=e^{-2\theta}
\]
This corresponds to 
\[
s=\kappa t
\]
and 
\begin{equation}
g(\phi s_{0})(s)=\phi(t)\kappa^{\frac{1}{2}}\Big(\frac{1+\kappa^{2}t^{2}}{1+t^{2}}\Big)^{-\frac{1}{2}}s_{0}(s)\label{eq:Klain Transformation Formula}
\end{equation}
Now the existence for some $q\geq0$ of an invariant generalized section
supported on $M$ (corresponding to $t_{0}=0$) would imply according
to \ref{prop:SectionsWithSubmanifoldSupport} the existence of a non-zero
invariant section over $M$ of $F=\Gamma_{M}^{-\infty,q}(X,K^{n,1})/\Gamma_{M}^{-\infty,q-1}(X,K^{n,1})=D^{*}(NM)\otimes Sym^{q}(NM)\otimes K^{n,1}|_{M}=D^{*}(NM)\otimes(NM)^{\otimes q}\otimes K^{n,1}|_{M}$
(for the last equality note that $NM$ is a line bundle).\\
Note that for $l\in M$, $N_{l}M=T_{l}X/T_{l}M=(l^{*}\otimes(V/l))/(l^{*}\otimes(l^{Q}/l))\backsimeq l^{*}\otimes(V/l^{Q})$,
where $l^{Q}$ is the $Q$-orthogonal complement of $l$, and $l\in M\iff l\subset l^{Q}$.
\\
Applying a pseudo-rotation (boost) by pseudo-angle $\theta$ fixing
$l$, the resulting transformation of the fiber of $F|_{l}$ is multiplication
by 
\[
\kappa\cdot\kappa^{q}\cdot\kappa^{1/2}
\]
for $\kappa=e^{-2\theta}$, which cannot equal $1$ for any $q$.
We conclude there are no invariant sections supported on the light
cone.
\end{proof}
When $k=1$, the Crofton fiber $E^{n,1}|_{\Lambda}$ is canonically
isomorphic (in particular, as $G$-equivariant bundles) to $D(V)^{n}\otimes D(\Lambda)^{*(n+1)}$:
\[
D(V/\Lambda)\otimes D(T_{\Lambda}Gr(n,n-1))=D(V/\Lambda)\otimes|\wedge^{top}((V/\Lambda)^{*}\otimes\Lambda)|=
\]
\[
=D(V/\Lambda)\otimes D((V/\Lambda)^{\otimes(n-1)})\otimes|\Lambda^{\wedge top}|=D(V/\Lambda)^{n}\otimes D(\Lambda)^{*}=
\]
\[
=D(V)^{n}\otimes D(\Lambda)^{*(n+1)}
\]
Let $\alpha$ be the anglular altitude on the sphere, and $z_{0}$
be the Euclidean section of the bundle $E^{n,1}$. The transformation
rule under the $G$-action for a boost $g_{\theta}$ by pseudo-angle
$\theta$ is therefore
\[
g_{*}(\phi z_{0})(\beta)=\phi(\alpha)\frac{|\cos2\beta|^{-\frac{n+1}{2}}}{|\cos2\alpha|^{-\frac{n+1}{2}}}s_{0}(\beta)
\]
or equivalently
\begin{equation}
g(\phi z_{0})(s)=\phi(t)\kappa^{-\frac{n+1}{2}}\Big(\frac{1+\kappa^{2}t^{2}}{1+t^{2}}\Big)^{\frac{n+1}{2}}z_{0}(s)\label{eq:Crofton Bundle formula}
\end{equation}
where $t=\tan(\frac{\pi}{4}-\alpha)$, $s=\tan(\frac{\pi}{4}-\beta)$,
$\beta=g_{\theta}\alpha$, $s=\kappa t$, and $\kappa=e^{-2\theta}$\\
\\
Let $f$ be a $G$-invariant generalized section of this bundle. When
restricted to an open orbit, such a section must be smooth (since
an open orbit is a homogeneous manifold for $G$). Therefore, on the
open orbits $f=C|\cos2\alpha|^{-\frac{n+1}{2}}z_{0}$, $C$ a locally
constant function on $Gr(V,n-1)$.\\
\\
In light of Corollary \ref{cor:The-map-is-an-injection}, we get
\begin{cor}
The space $\Gamma^{-\infty}(Gr(n,n-1),E^{n,1})^{G}$ is at most 2-dimensional
\end{cor}
We will now turn to constructing two independent sections of this
space, proving is it in fact $2$-dimensional. Let us first remark
that applying Proposition \ref{prop:SectionsWithSubmanifoldSupport}
for this manifold (this time $T_{\Lambda}M=(\Lambda/\Lambda^{Q})^{*}\otimes(V/\Lambda)$),
one can see that an invariant generalized section supported on the
light cone can exist only if 
\[
q+1-\frac{n+1}{2}=0\iff n=2q+1
\]
where $q$ is the order of the section (as a differential operator).
We will show that such sections do indeed exist. 
\begin{prop}
\label{cor:light_cone_only_odd_n}$\dim\Gamma^{-\infty}(E^{n,1})^{G}=2$.
For odd $n$, there is a one dimensional subspace of generalized sections
supported on the light cone. For even $n$, none are supported on
the light cone.\end{prop}
\begin{proof}
The sections are associated with the generalized functions on $Gr(n,n-1)$
constructed in \ref{sec:Some-generalized-functions}. \\
According to equations \ref{eq:group_action-grassmann} and \ref{eq:Crofton Bundle formula},
they are given (after a Euclidean trivialization) by $f_{n,n-1,\lambda}^{\pm}$
with $\lambda=-\frac{n+1}{2}$. The support properties follow immediately
from the corresponding properties for $f_{\lambda}^{\pm}$. 
\end{proof}
Those sections will be denoted $f_{n,1}^{\pm}$.\\
Let us write explicit formulas for those sections in some dimensions:
For $n=3$, the cone-symmetric section $f_{3,1}^{+}$ (after rescaling)
is given by 
\[
\phi(\alpha,\psi)d\sigma\mapsto\int_{\epsilon=0}^{\frac{\pi}{4}}\int_{\psi=0}^{2\pi}\frac{\phi(\frac{\pi}{4}+\epsilon,\psi)+\phi(\frac{\pi}{4}-\epsilon,\psi)-2\phi(\frac{\pi}{4},\psi)}{|\sin2\epsilon|^{2}}\sin(\epsilon+\frac{\pi}{4})d\epsilon d\psi+
\]
\[
+\sqrt{2}I_{0}(-2)\int_{\psi=0}^{2\pi}\phi(\frac{\pi}{4},\psi)d\psi
\]
and the cone-antisymmetric section $f_{3,1}^{-}$ is given by 
\[
\phi(\alpha,\psi)d\sigma\mapsto\frac{\partial}{\partial\alpha}\Bigg|_{\alpha=\frac{\pi}{4}}(\sin\alpha\int_{S^{1}}d\psi\phi(\alpha,\psi))
\]
For higher odd values of $n$, the cone-antisymmetric section is given
by 
\[
\phi(\alpha,\psi)d\sigma\mapsto\frac{\partial^{m}}{\partial\alpha^{m}}\Bigg|_{\alpha=\frac{\pi}{4}}(\sin^{n-2}\alpha\int_{M}\phi(\alpha,\psi)d\psi)+\text{lower order derivatives}
\]
where $m=\frac{n-1}{2}$.

\subsubsection{Case of general $k$}

Denote $X=Gr(V,k)$, $M$ the set of $Q$-degenerate subspaces. 
\begin{prop}
There are no $G$-invariant sections over $M$ of the bundle with
fiber over $\Lambda$ equal to\textup{ $D^{*}(N_{\Lambda}M)\otimes Sym^{q}(N_{\Lambda}M)\otimes K^{n,k}|_{\Lambda}$
.}\end{prop}
\begin{proof}
Fix $\Lambda\in M$ touching the light cone $C$ along the line $l=\Lambda\cap\Lambda^{Q}$.
Denote also $\Omega=\Lambda+\Lambda^{Q}=l^{Q}$. Write $N_{\Lambda}M=T_{\Lambda}X/T_{\Lambda}M$.
Then 
\[
N_{\Lambda}M=l^{*}\otimes(V/\Omega)
\]
Thus as in the case $k=1$, for $g=g_{\theta}\in\text{Stab}(\Lambda)$,
the action on $D^{*}(N_{\Lambda}M)\otimes Sym^{q}(N_{\Lambda}M)\otimes K^{n,k}|_{\Lambda}$
is by multiplication by $\kappa^{q+1}\kappa^{1/2}$ where $\kappa=e^{-2\theta}$.
So again by Proposition \ref{prop:SectionsWithSubmanifoldSupport},
there are no invariant generalized sections of $K^{n,k}$ supported
on the light cone. 
\end{proof}
Therefore by Proposition \ref{prop:KlainSectionsAre2Dimensional},
$\dim\Gamma^{-\infty}(X,K^{n,k})^{G}=2$.
\begin{prop}
\label{prop:noConeSupportGeneralK}$\dim\Gamma^{-\infty}(E^{n,k})^{G}=2$
for all $1\leq k\leq n-1$. For odd $n$, there is a one dimensional
subspace of generalized sections supported on the light cone. For
even $n$, none are supported on the light cone.\end{prop}
\begin{proof}
Again by Corollary \ref{cor:The-map-is-an-injection}, $\dim\Gamma^{-\infty}(E^{n,k})^{G}\leq2$.
Let us find two independent sections explicitly. This time $\Lambda\in Gr(V,n-k)$
and 

\[
E_{\Lambda}=D(V/\Lambda)\otimes D(T_{\Lambda}Gr(V,n-k))=
\]
\[
=D(V)\otimes D(\Lambda)^{*}\otimes D(\Lambda^{*}\otimes V/\Lambda)=
\]
\[
=D(V)\otimes D^{*}(\Lambda)\otimes D(V)^{n-k}\otimes D^{*}(\Lambda)^{n}=
\]
\[
=D(V)^{n-k+1}\otimes D(\Lambda)^{*(n+1)}
\]
So similarly to the case $k=1$, the invariant sections $f_{n,k}^{\pm}$
of $E^{n,k}$ are given, after the Euclidean trivialization, by $f_{n,n-k,\lambda}^{\pm}$,
with $\lambda=-\frac{n+1}{2}$. 
\end{proof}
For even values of $n$, we will also use the basis $f_{n,k}^{S},f_{n,k}^{T}$
corresponding to $f_{\lambda}^{S},f_{\lambda}^{T}$. \\
\\
Recall that for $\mu\in\M^{\infty}(Gr(n,k))^{SO(n-1)}$ such that
$\mu=\phi(\alpha)d\Lambda$ where $d\Lambda$ is the unique $SO(n)$-invariant
probability measure on $Gr(n,n-k)$ we have 
\[
f_{n,k}^{\pm}(\mu)=f{}_{\lambda}^{\pm}(\phi(\alpha)g_{n,n-k}(\alpha)d\alpha)
\]
where $g_{n,n-k}(\alpha)=C_{n,k}\sin^{n-k-1}\alpha\cos^{k-1}\alpha$
and $\lambda=-\frac{n+1}{2}$. From now on, we renormalize $f_{n,k}^{\pm}$
so that $C_{n,k}=1$.
\begin{thm}
\label{cor:dimensionAtMost2}For all $1\leq k\leq n-1$, $\dim Val_{k}^{ev,-\infty}(\R^{n})^{SO^{+}(n-1,1)}=2$. \end{thm}
\begin{proof}
According to Proposition \ref{prop:KlainExtension}, $\Big(Val_{k}^{ev,-\infty}(V)\Big)^{G}$
is naturally a subspace of $\Big(\Gamma^{-\infty}(K^{n,k})\Big)^{G}$.
In particular, $\dim\Big(Val_{k}^{ev,-\infty}(V)\Big)^{G}\leq2$.
Then by Proposition \ref{prop:CroftonExtension}, $\text{Ker}\Big(Cr_{n-k}:\Gamma^{-\infty}(E^{n,k})\to Val_{k}^{ev,-\infty}(V)\Big)\subset\text{Ker}T_{n-k,k}$
so by Corollary \ref{cor:The-map-is-an-injection} one has $\dim Val_{k}^{ev,-\infty}(V)^{G}\geq2$.
Thus, we get equality.
\end{proof}
It follows that every $SO^{+}(n-1,1)$-invariant continuous valuation
$\phi\in Val_{k}(V)$ is determined by its uniquely-defined $SO^{+}(n-1,1)$-invariant
generalized Crofton measure.

\section{\label{sec:5 real non existence}The non-existence of even Lorentz-invariant
valuations for $1\leq k\leq n-2$ }

We now proceed to show that the generalized valuations $\phi=Cr(f_{n,k}^{\pm})$
corresponding to the sections $f_{n,k}^{\pm}\in\Gamma^{-\infty}(E^{n,k})^{G}$
that we found, are not continuous valuations. In fact, we will show
those valuations cannot be extended by continuity to the double cone.
By Lemma \ref{croftonMeasureDefinesValuationOnSmoothBodies}, it follows
that for an $SO(n-1)$-invariant smooth unconditional body $K^{n}$
with $k$-support function $h_{k}(\alpha;K)$, those valuations are
given by 
\[
\phi(K^{n})=f_{\lambda}^{\pm}(h_{k}(\alpha;K)g_{n,n-k}(\alpha)d\alpha)
\]
with $\lambda=-\frac{n+1}{2}$. Then by \ref{prop:disjoint_singular_supports},
the same formula holds as long as $h_{k}(\alpha;K)$ is smooth near
the light cone.

\subsection{Computations related to the double cone}

In the following, $C\subset\R^{2}$ is the unit ball of the $l_{1}$
norm. We will write $h_{k}(\alpha)=h_{k}(\alpha;C^{k+1})$ (where
$-\frac{\pi}{2}\leq\alpha\leq\frac{\pi}{2}$ is the angle between
the normal to the hyperplane to which $C^{k+1}$ is projected, and
the space-like coordinate hyperplane). It can be computed as follows:
fix $u=(\cos\alpha,0,...,0,\sin\alpha)$ the normal to the hyperplane,
and $v=(\cos\beta w,\sin\beta)$, $w\in S^{k-1}$. The surface area
measure of $C^{k+1}$ is $\sigma_{C^{k+1}}(v)=\delta_{\frac{\pi}{4}}(\beta)+\delta_{-\frac{\pi}{4}}(\beta)$
and 
\[
h_{k}(\alpha)=T_{k}(\sigma_{C^{k+1}}(\beta))(\alpha)=\int_{S^{k}}(\delta_{\frac{\pi}{4}}(\beta)+\delta_{-\frac{\pi}{4}}(\beta))|\langle u,v\rangle|\cos^{k-1}\beta d\beta d\sigma_{k-1}(w)
\]
If $k>2$, we take $0-\frac{\pi}{2}\leq\phi\leq\frac{\pi}{2}$ to
be the elevation angle of $w\in S^{k-1}$. If $k=2$, $-\pi\leq\phi\leq\pi$.
Let us write $a_{k}\leq\phi\leq b_{k}$ for both cases. Then 
\[
h_{k}(\alpha)=C_{k}\int_{a_{k}}^{b_{k}}\int_{-\frac{\pi}{2}}^{\frac{\pi}{2}}(\delta_{\frac{\pi}{4}}(\beta)+\delta_{-\frac{\pi}{4}}(\beta))|\sin\beta\sin\alpha+\cos\beta\cos\alpha\sin\phi|\cos^{k-1}\beta\cos^{k-2}\phi d\beta d\phi=
\]
\[
=\frac{C_{k}}{2^{k/2}}\int_{a_{k}}^{b_{k}}(|\sin\alpha+\cos\alpha\sin\phi|+|\sin\alpha-\cos\alpha\sin\phi|)\cos^{k-2}\phi d\phi=
\]
\[
=\frac{2C_{k}}{2^{k/2}}\int_{a_{k}}^{b_{k}}|\sin\alpha-\cos\alpha\sin\phi|\cos^{k-2}\phi d\phi
\]
\[
=\Bigg\{\begin{array}{c}
h_{k}^{+}(\alpha),\,\frac{\pi}{4}\leq\alpha\leq\frac{\pi}{2}\\
h_{k}^{-}(\alpha),\,0\leq\alpha\leq\frac{\pi}{4}
\end{array}
\]
Denoting $A_{k}=\int_{-\pi/2}^{\pi/2}\cos^{k-2}\phi d\phi$, and replacing
$C_{2}$ by $2C_{2}$ for $k=2$, we get 
\[
h_{k}^{+}(\alpha)=\frac{2C_{k}}{2^{k/2}}A_{k}\sin\alpha
\]
and 
\[
h_{k}^{-}(\alpha)=\frac{2C_{k}}{2^{k/2}}\Bigg(A_{k}\sin\alpha-2\sin\alpha\int_{\arcsin\tan\alpha}^{\pi/2}\cos^{k-2}\phi d\phi+\frac{2}{k-1}\frac{(\cos2\alpha)^{\frac{k-1}{2}}}{(\cos\alpha)^{k-2}}\Bigg)=
\]
\[
=h_{k}^{+}(\alpha)+\frac{2C_{k}}{2^{k/2}}\Bigg(\frac{2}{k-1}\frac{(\cos2\alpha)^{\frac{k-1}{2}}}{(\cos\alpha)^{k-2}}-2\sin\alpha\int_{\tan\alpha}^{1}(1-t^{2})^{\frac{k-3}{2}}dt\Bigg)=
\]
with the exception 
\[
h_{1}(\alpha)=\frac{1}{\sqrt{2}}(|\sin(\alpha+\frac{\pi}{4})|+|\cos(\alpha+\frac{\pi}{4})|)=\max(|\sin\alpha|,|\cos\alpha|)
\]
For $\epsilon>0$ and every $n$ define the $\epsilon-$ stretching
of $\R^{n}$, $S_{\epsilon}$ to be the diagonal $n\times n$ matrix
$c_{\epsilon}\text{diag}(1,...,1,\tan(\frac{\pi}{4}+\epsilon))$ where
where $c_{\epsilon}\to1$ as $\epsilon\to0$ will be specified shortly.
In the following, we will denote $\eta=\tan(\frac{\pi}{4}+\epsilon)$.
We replace the double cone with its $\epsilon-$ stretching $C_{n,\epsilon}=S_{\epsilon}C^{n}$,
and take $c_{\epsilon}$ such that $h_{k}(\frac{\pi}{4};C_{\epsilon})=\eta h_{k}(\frac{\pi}{4};C)$.
We will write in the following $h_{k,\epsilon}(\alpha)=h_{k}(\alpha;C_{n,\epsilon})$,
omitting $\epsilon$ when $\epsilon=0$. Again for all $k\leq n-1$
\[
h_{k,\epsilon}(\alpha)=c_{\epsilon}C_{k}\int_{a_{k}}^{b_{k}}\int_{-\frac{\pi}{2}}^{\frac{\pi}{2}}(\delta_{\frac{\pi}{4}+\epsilon}(\beta)+\delta_{-\frac{\pi}{4}-\epsilon}(\beta))|\sin\beta\sin\alpha+\cos\beta\cos\alpha\sin\phi|\cos^{k-1}\beta\cos^{k-2}\phi d\beta d\phi
\]
Let us write 
\[
h_{k,\epsilon}(\alpha)=\Bigg\{\begin{array}{c}
h_{k,\epsilon}^{+}(\alpha),\,\alpha\geq\frac{\pi}{4}-\epsilon\\
h_{k,\epsilon}^{-}(\alpha),\,0\leq\alpha\leq\frac{\pi}{4}-\epsilon
\end{array}
\]
where for $k\geq2$ (again the definition of $C_{k}$ for $k=2$ is
twice the definition of $C_{k}$ for $k\geq3$) 
\[
h_{k,\epsilon}^{+}(\alpha)=\frac{2C_{k}}{2^{k/2}}A_{k}\eta\cdot\sin\alpha
\]
and
\[
h_{k,\epsilon}^{-}(\alpha)=\frac{2C_{k}}{2^{k/2}}\Bigg(\eta\sin\alpha\Big(A_{k}-2\int_{\arcsin(\eta\tan\alpha)}^{\pi/2}\cos^{k-2}\phi d\phi\Big)+
\]
\[
+\frac{2}{k-1}\cos\alpha(1-\eta^{2}\tan^{2}\alpha)^{\frac{k-1}{2}}\Bigg)=
\]
\[
=h_{k,\epsilon}^{+}(\alpha)+\frac{2C_{k}}{2^{k/2}}\Bigg(\frac{2}{k-1}\cos\alpha(1-\eta^{2}\tan^{2}\alpha)^{\frac{k-1}{2}}-2\eta\sin\alpha\int_{\eta\tan\alpha}^{1}(1-t^{2})^{\frac{k-3}{2}}dt\Bigg)
\]
While 
\[
h_{1}(\alpha;C_{n,\epsilon})=\max(\eta|\sin\alpha|,|\cos\alpha|)
\]
By rescaling the bodies, and since we will be only considering a single
value of $k$ at a time, we may assume in the subsequent computations
that $\frac{2C_{k}}{2^{k/2}}=1$ for all $h_{k}$. 
\begin{rem}
\label{Remark evaluate valuation on stretched cone}In this computation,
$\alpha$ is the angle between the normal in $\R^{k+1}$ to the hyperplane
to which we project, and the space coordinate hyperplane; The value
of the even $k-$homogeneous cone-symmetric/antisymmetric valuation
in $\R^{n}$ on $C_{n,\epsilon}$ when $\epsilon\neq0$ is given by
$f_{-\frac{n+1}{2}}^{\pm}(h_{k,\epsilon}(\alpha)g_{n,n-k}(\alpha)d\alpha)$
by Proposition \ref{prop:disjoint_singular_supports} since the singular
support (in fact, the support) of the surface area measure of $C_{n,\epsilon}$
is disjoint from the light cone.
\end{rem}

\begin{rem}
\label{Remark smooth extension of support function}We observe for
the following that $h_{k}^{+}$, admits a real analytic extension
to $S^{1}$, and if $k$ is odd then also $h_{k}^{-}$ admits a real
analytic extension to $\alpha\in(-\frac{\pi}{2},\frac{\pi}{2})$.
The same holds for $h_{k,\epsilon}^{\pm}$, and in the corresponding
cases it holds in the $C^{\infty}$ topology that 
\[
\lim_{\epsilon\to0^{\pm}}h_{k,\epsilon}^{\pm}=h_{k}^{\pm}
\]
It follows that for any continuous valuation $\phi$ with generalized
Crofton measure $f_{n,k}^{\pm}$, one may write 
\[
\phi(C^{n})=\lim_{\eps\to0^{+}}\phi(C_{n,\eps})=\lim_{\eps\to0^{+}}f_{-\frac{n+1}{2}}^{\pm}(h_{k,\epsilon}^{+}(\alpha)g_{n,n-k}(\alpha))=f_{-\frac{n+1}{2}}^{\pm}(h_{k}^{+}(\alpha)g_{n,n-k}(\alpha))
\]
and if $n$ is odd then also 
\[
\phi(C^{n})=f_{-\frac{n+1}{2}}^{\pm}(h_{k}^{-}(\alpha)g_{n,n-k}(\alpha))
\]

\end{rem}

\subsection{Applying the generalized valuations to the double cone}
\begin{prop}
(Reduction to $k=n-2$) If for every $n\geq3$ there exists no continuous
even $G$-invariant $(n-2)$-homogeneous valuation, then there exists
no continuous even $G$-invariant $j$-homogeneous valuation for $j<n-2$.\end{prop}
\begin{proof}
Let $\phi\in Val_{j}^{+}(\R^{n})^{SO^{+}(n-1,1)}$ be such a valuation
with $j<n-2$. By our assumption, if $\Lambda$ is any $(n-2)$-subspace
s.t. $Q|_{\Lambda}$ has mixed signature then $\phi|_{\Lambda}=0$.
Since every $j$-dimensional subspace is contained in some $\Lambda$
as above, we conclude that $Kl_{j}(\phi)=0$, and therefore $\phi=0$.
\end{proof}
Thus we may assume from now on that $k=n-2$, and prove non-extendibility
of the corresponding valuations.
\begin{prop}
\label{prop:odd N LightCone}(Odd $n$, light cone support). For odd
$n$, an $n-2$-homogeneous even valuation $\phi$ on $\R^{n}$ having
generalized Crofton measure $f\in\Gamma^{-\infty}(E^{n,k})^{G}$ supported
on the light cone, cannot be extended by continuity to all $SO(n-1)$-invariant
compact convex bodies.\end{prop}
\begin{proof}
Assume, on the contrary, that this can be accomplished. We will show
that $\phi$ does not extend to the double cone by continuity. Recall
from \ref{prop:noConeSupportGeneralK} that a valuation $\phi$ as
above can occur only for odd $n$, and by Remark \ref{Remark Cone Symmetric And Antisymmetric},
it is cone-symmetric if $n\equiv1\mod4$ and cone-antisymmetric otherwise.
By Remark \ref{Remark evaluate valuation on stretched cone}, we may
evaluate the valuation on $C_{n,\epsilon}$ by $\phi(C_{n,\epsilon})=f(h_{k}(\alpha;C_{n,\epsilon})g_{n,n-k})$.
Therefore, 
\[
\phi(C^{n})=\lim_{\epsilon\to0}f(h_{k}(\alpha;C_{n,\epsilon})g_{n,n-k})
\]
Write
\[
f(hg_{n,n-k})=\sum_{j=0}^{m}c_{j}h^{(j)}(\frac{\pi}{4})
\]
with $m=\frac{n-1}{2}$ (note that the derivatives of $g_{n,n-k}$
are now incorporated into the coefficients $c_{j}$). Note that $c_{m}\neq0$
since $g_{n,n-k}(\frac{\pi}{4})\neq0$. We will show that the limits
$\lim_{\epsilon\to0^{+}}f(h_{k,\epsilon}(\alpha,C)g_{n,n-k})$ and
$\lim_{\epsilon\to0^{-}}f(h_{k,\epsilon}(\alpha,C)g_{n,n-k})$ are
finite and different from one another, thus arriving at a contradiction.
Equivalently, since $\lim_{\epsilon\to0^{+}}h_{k,\epsilon}=h_{k}^{+}$
in the $C^{\infty}[-\frac{3\pi}{8},\frac{3\pi}{8}]$ topology, we
will show that 
\[
\lim_{\epsilon\to0^{-}}\Big(f(h_{k,\epsilon}^{+}(\alpha,C)g_{n,n-k})-f(h_{k,\epsilon}^{-}(\alpha,C)g_{n,n-k})\Big)=\lim_{\epsilon\to0^{-}}f((h_{k,\epsilon}^{+}(\alpha,C)-h_{k,\epsilon}^{-}(\alpha,C))g_{n,n-k})\Big)
\]
is non-zero. Denote $v_{\epsilon}(\alpha)=h_{k,\epsilon}^{+}(\alpha,C)-h_{k,\epsilon}^{-}(\alpha,C)$
and $u_{\epsilon}(\alpha)=(\sin\alpha)^{-1}v_{\epsilon}(\alpha)$.
\\
Consider first the case $n>3$. Then

\[
u_{\epsilon}(\alpha)=2\eta\int_{\eta\tan\alpha}^{1}(1-t^{2})^{\frac{k-3}{2}}dt-\frac{2}{k-1}\cot\alpha(1-\eta^{2}\tan^{2}\alpha)^{\frac{k-1}{2}}
\]
where as before $\eta=\tan(\frac{\pi}{4}+\epsilon)$. It suffices
to prove that $\lim_{\epsilon\to0^{-}}u_{\epsilon}^{(j)}(\frac{\pi}{4})=0$
for $j\leq m-1$, and is non-zero for $j=m$. Indeed, $\lim_{\epsilon\to0^{-}}u_{\epsilon}(\frac{\pi}{4})=0$,
and 
\[
u_{\epsilon}'(\alpha)=\frac{2}{k-1}\frac{(1-\eta^{2}\tan^{2}\alpha)^{\frac{k-1}{2}}}{\sin^{2}\alpha}
\]
Since $k=n-2$, the numerator is a polynomial in $\tan^{2}\alpha$
with coefficients depending on $\epsilon$, and we conclude that $u_{\epsilon}'\to\frac{2}{k-1}\frac{1}{\sin^{2}\alpha}(1-\tan^{2}\alpha)^{\frac{k-1}{2}}$
in $C^{\infty}[-\frac{3\pi}{8},\frac{3\pi}{8}]$. Since 
\[
(1-\tan^{2}\alpha)^{\frac{k-1}{2}}=(1-(1+4(\alpha-\frac{\pi}{4})+o(\alpha-\frac{\pi}{4})))^{\frac{k-1}{2}}=(-4)^{\frac{k-1}{2}}(\alpha-\frac{\pi}{4})^{\frac{k-1}{2}}+o((\alpha-\frac{\pi}{4})^{\frac{k-1}{2}})
\]
and $\frac{k-1}{2}=\frac{n-3}{2}=m-1$, it follows that 
\[
\Bigg((1-\tan^{2}\alpha)^{m-1}\Bigg)^{(m-1)}(\frac{\pi}{4})=(-4)^{m-1}(m-1)!
\]
implying the claim.

Now assume $n=3$ so $k=1$. Then $v_{\epsilon}(\alpha)=\eta\cos\alpha-\sin\alpha$,
where again $\eta=\tan(\frac{\pi}{4}+\epsilon)$. Since $\lim_{\epsilon\to0^{-}}v_{\epsilon}(\frac{\pi}{4})=0$
while $\lim_{\epsilon\to0^{-}}v_{\epsilon}'(\frac{\pi}{4})=\lim_{\epsilon\to0^{-}}-\sin\alpha\eta-\cos\alpha=-\sqrt{2}\neq0$,
the claim follows. \end{proof}
\begin{rem}
\label{Remark finite one sided limits}We note for the following that
for odd values of $n$, both $\lim_{\epsilon\to0^{+}}f(h_{k,\epsilon}(\alpha,C)g_{n,n-k})$
and $\lim_{\epsilon\to0^{-}}f(h_{k,\epsilon}(\alpha,C)g_{n,n-k})$
are finite, where $f$ is the unique $G$-invariant generalized Crofton
measure supported on the light cone.\end{rem}
\begin{prop}
\label{prop :integral computation}(Odd $n$) For odd $n$, no $n-2$-homogeneous
valuation $\phi\in Val_{n-2}^{+}(\R^{n})^{G}$ exists.\end{prop}
\begin{proof}
Denote $k=n-2$. Assume first that $\phi$ is either pure cone-symmetric
or cone-antisymmetric, according to $n\mod4$, such that it is not
supported on the light cone. \\
\\
First, assume $n\equiv1\mod4$, so $n\geq5$ and $k\geq3$. Then $\frac{n+1}{2}$
is odd, and $\phi=Cr(f_{n,k}^{-})$. 
\[
\phi(C_{n})=\lim_{\epsilon\to0^{+}}\phi(C_{n,\epsilon})=\lim_{\epsilon\to0^{+}}f_{-\frac{n+1}{2}}^{-}(h_{k}(\alpha;C_{n,\epsilon})g_{n,n-k})
\]
Note that $h_{k}(\alpha;C_{n,\epsilon})=C\eta\sin\alpha$ near $\alpha=\frac{\pi}{4}$,
and so all derivatives at $\alpha=\frac{\pi}{4}$ of $h_{k}(\alpha;C_{n,\epsilon})g_{n,n-k}$
converge to a finite limit as $\epsilon\to0^{+}$. Write for an arbitrary
function $H$ on $S^{1}$,
\[
N_{-}(\alpha;H)=\frac{H(\alpha)-H(\frac{\pi}{2}-\alpha)-2(H'(\frac{\pi}{4})(\alpha-\frac{\pi}{4})+\frac{1}{3!}H^{(3)}(\frac{\pi}{4})(\alpha-\frac{\pi}{4})^{3}+...+\frac{1}{(2m+1)!}H^{(2m+1)}(\frac{\pi}{4})(\alpha-\frac{\pi}{4})^{2m+1})}{|\cos2\alpha|^{\frac{n+1}{2}}}
\]
 where $m=\frac{n-1}{4}$. Denote $H_{\epsilon}(\alpha)=h_{k}(\alpha;C_{n,\epsilon})g_{n,n-k}(\alpha)$.
We will show that the integral 
\[
I_{-}(H_{\epsilon})=\int_{0}^{\frac{\pi}{4}}N_{-}(\alpha;H_{\epsilon})d\alpha
\]
which equals $\phi(C_{n,\epsilon})$ up to bounded summands, diverges
as $\epsilon\to0^{+}$. Then
\[
I_{-}(H_{\epsilon})=\int_{0}^{\frac{\pi}{4}-\epsilon}\frac{h_{k,\epsilon}^{-}(\alpha)g_{n,n-k}(\alpha)-h_{k,\epsilon}^{+}(\alpha)g_{n,n-k}(\alpha)}{|\cos2\alpha|^{\frac{n+1}{2}}}d\alpha+\int_{0}^{\frac{\pi}{4}}N_{-}(\alpha;h_{k,\epsilon}^{+}(\alpha)g_{n,n-k}(\alpha))d\alpha
\]
Now the second integral is bounded (uniformly in $\epsilon)$, for
instance by $C|\int_{0}^{\frac{\pi}{4}}N_{-}(\alpha;h_{k}^{+}(\alpha)g_{n,n-k}(\alpha))d\alpha|$.
\\
\\
We will show that the first summand is unbounded. Calculate first
that 
\[
\frac{d}{d\alpha}\Big(\frac{h_{k,\epsilon}^{-}(\alpha)-h_{k,\epsilon}^{+}(\alpha)}{\sin\alpha}\Big)=2\frac{d}{d\alpha}\Bigg(\frac{2}{k-1}\cot\alpha(1-\eta^{2}\tan^{2}\alpha)^{\frac{k-1}{2}}-\Bigg(\int_{\eta\tan\alpha}^{1}(1-t^{2})^{\frac{k-3}{2}}dt\Bigg)\eta\Bigg)=
\]
\[
=\frac{\eta^{2}}{\cos^{2}\alpha}(1-\eta^{2}\tan^{2}\alpha)^{\frac{k-3}{2}}-(1-\eta^{2}\tan^{2}\alpha)^{\frac{k-3}{2}}\Big(\frac{2}{k-1}\frac{1-\eta^{2}\tan^{2}\alpha}{\sin^{2}\alpha}+\frac{\eta^{2}}{\cos^{2}\alpha}\Big)=
\]
\begin{equation}
=-\frac{2}{k-1}\frac{(1-\eta^{2}\tan^{2}\alpha)^{\frac{k-1}{2}}}{\sin^{2}\alpha}\label{eq:amazing_derivative}
\end{equation}
which is negative. Since $h_{k,\epsilon}^{+}(\frac{\pi}{4}-\epsilon)=h_{k,\epsilon}^{-}(\frac{\pi}{4}-\epsilon)$,
it follows that $h_{k,\epsilon}^{-}(\alpha)-h_{k,\epsilon}^{+}(\alpha)>0$
in $(0,\frac{\pi}{4}-\epsilon)$. Now
\[
\int_{0}^{\frac{\pi}{4}-\epsilon}\frac{h_{k,\epsilon}^{-}(\alpha)g_{n,n-k}(\alpha)-h_{k,\epsilon}^{+}(\alpha)g_{n,n-k}(\alpha)}{|\cos2\alpha|^{\frac{n+1}{2}}}d\alpha\geq C_{n}+c_{n}\int_{\frac{\pi}{5}}^{\frac{\pi}{4}-\epsilon}\frac{h_{k,\epsilon}^{-}(\alpha)-h_{k,\epsilon}^{+}(\alpha)}{(\frac{\pi}{4}-\alpha)^{\frac{n+1}{2}}}d\alpha
\]
\[
\int_{\frac{\pi}{5}}^{\frac{\pi}{4}-\epsilon}\frac{h_{k,\epsilon}^{-}(\alpha)-h_{k,\epsilon}^{+}(\alpha)}{(\frac{\pi}{4}-\alpha)^{\frac{n+1}{2}}}d\alpha=
\]
 
\[
\geq c_{n}'\int_{\frac{\pi}{5}}^{\frac{\pi}{4}-\epsilon}\frac{1}{(\frac{\pi}{4}-\alpha)^{\frac{n+1}{2}}}\frac{h_{k,\epsilon}^{-}(\alpha)-h_{k,\epsilon}^{+}(\alpha)}{\sin\alpha}d\alpha
\]
Now integrate by parts: we integrate $(\frac{\pi}{4}-\alpha)^{-\frac{n+1}{2}}$
and differentiate the other term. The boundary term is bounded uniformly
in $\epsilon$, and we already computed the derivative of $\frac{h_{k,\epsilon}^{-}(\alpha)-h_{k,\epsilon}^{+}(\alpha)}{\sin\alpha}$
in equation \ref{eq:amazing_derivative}. The resulting integral thus
equals 
\[
=c_{n,k}\int_{\frac{\pi}{5}}^{\frac{\pi}{4}-\epsilon}\frac{(1-\eta^{2}\tan^{2}\alpha)^{\frac{k-1}{2}}d\alpha}{(\frac{\pi}{4}-\alpha)^{\frac{n-1}{2}}\sin^{2}\alpha}\geq c_{n,k}'\int_{\frac{\pi}{5}}^{\frac{\pi}{4}-\epsilon}\frac{(1-\eta^{2}\tan^{2}\alpha)^{\frac{k-1}{2}}d\alpha}{(\frac{\pi}{4}-\alpha)^{\frac{n-1}{2}}}
\]
Now $1-\eta^{2}\tan^{2}\alpha\geq\frac{1}{4}(\alpha_{\epsilon}-\alpha)$
so the integral is bounded from below by 
\[
c_{n,k}'\int_{\frac{\pi}{5}}^{\frac{\pi}{4}-\epsilon}\frac{(\alpha_{\epsilon}-\alpha)^{\frac{k-1}{2}}d\alpha}{(\frac{\pi}{4}-\alpha)^{\frac{n-1}{2}}}=c_{n,k}'\int_{0}^{\frac{\pi}{4}-\epsilon-\frac{\pi}{5}}\frac{t{}^{\frac{k-1}{2}}dt}{(\epsilon+t)^{\frac{n-1}{2}}}\geq c_{n,k}'\int_{0}^{\frac{\pi}{100}}\frac{t{}^{\frac{k-1}{2}}dt}{(\epsilon+t)^{\frac{n-1}{2}}}
\]
Finally, the limit 
\[
\lim_{\epsilon\to0^{+}}\int_{0}^{\frac{\pi}{100}}\frac{t{}^{\frac{k-1}{2}}dt}{(\epsilon+t)^{\frac{n-1}{2}}}=\infty
\]
is infinite. Thus $I_{-}(H_{\epsilon})$ is unbounded as $\epsilon\to0^{+}$,
i.e $\phi(C_{n,\epsilon})\to\infty$.\\
\\
Now assume $n\equiv3\mod4$ and $n\geq7$, so $k\geq5$ and $\phi$
corresponds to $f_{n,k}^{+}$. For an arbitrary function $H$ on $S^{1}$
define $N_{+}(\alpha;H)$ by 
\[
N_{+}(\alpha;H)=\frac{H(\alpha)+H(\frac{\pi}{2}-\alpha)-2(H(\frac{\pi}{4})+\frac{1}{2!}H^{(2)}(\frac{\pi}{4})(\alpha-\frac{\pi}{4})^{2}+...+\frac{1}{(2m)!}H^{(2m)}(\frac{\pi}{4})(\alpha-\frac{\pi}{4})^{2m})}{|\cos2\alpha|^{\frac{n+1}{2}}}
\]
where $m=\frac{n-3}{4}$. Exactly as before, the integral 
\[
I_{+}(H_{\epsilon})=\int_{0}^{\frac{\pi}{4}}N_{+}(\alpha;H_{\epsilon})d\alpha
\]
is unbounded as $\epsilon\to0^{+}$, i.e $\phi(C_{n,\epsilon})\to\infty$.
\\
\\
Let us compute separately the case of $k=1$ and $n=3$. \\
Then
\[
I_{-}(H_{\epsilon})=\int_{0}^{\frac{\pi}{4}-\epsilon}\frac{\cos\alpha g_{n,n-1}(\alpha)-\eta\sin\alpha g_{n,n-1}(\alpha)}{|\cos2\alpha|^{\frac{n+1}{2}}}d\alpha+\int_{0}^{\frac{\pi}{4}}N_{-}(\alpha;\eta\sin\alpha g_{n,n-1}(\alpha))d\alpha
\]
where
\[
N_{-}(\alpha;H)=\frac{H(\alpha)-H(\frac{\pi}{2}-\alpha)-2H'(\frac{\pi}{4})(\alpha-\frac{\pi}{4})}{|\cos2\alpha|^{\frac{n+1}{2}}}
\]
Now the second integral is bounded (uniformly in $\epsilon)$, for
instance by $2|\int_{0}^{\frac{\pi}{4}}N_{-}(\alpha;\sin\alpha g_{n,n-1}(\alpha))d\alpha|$.
The first integrand is non-negative, and since $g_{n,n-1}(\alpha)\geq c_{n}$
for $\alpha\in[\frac{\pi}{10},\frac{\pi}{4}]$ while $\cos2\alpha\leq c|\alpha-\frac{\pi}{4}|$
in that interval, we get
\[
\int_{0}^{\frac{\pi}{4}-\epsilon}\frac{\cos\alpha g_{n,n-1}(\alpha)-\eta\sin\alpha g_{n,n-1}(\alpha)}{|\cos2\alpha|^{\frac{n+1}{2}}}d\alpha\geq c\int_{\frac{\pi}{10}}^{\frac{\pi}{4}-\epsilon}\frac{\cos\alpha-\eta\sin\alpha}{(\frac{\pi}{4}-\alpha)^{\frac{n+1}{2}}}d\alpha\geq
\]
\[
\geq c\int_{0}^{\frac{\pi}{4}-\epsilon}\frac{\cos\alpha-\eta\sin\alpha}{(\frac{\pi}{4}-\alpha)^{\frac{n+1}{2}}}d\alpha
\]
The function $\cos\alpha-\eta\sin\alpha$ is decreasing and concave
for $0\leq\alpha\leq\frac{\pi}{4}-\epsilon$, so $\cos\alpha-\eta\sin\alpha\geq1-\frac{\alpha}{\frac{\pi}{4}-\epsilon}$
for $0\leq\alpha\leq\frac{\pi}{4}-\epsilon$. Therefore 
\[
\int_{0}^{\frac{\pi}{4}-\epsilon}\frac{\cos\alpha-\eta\sin\alpha}{(\frac{\pi}{4}-\alpha)^{\frac{n+1}{2}}}d\alpha\geq\frac{1}{\frac{\pi}{4}-\epsilon}\int_{0}^{\frac{\pi}{4}-\epsilon}(\frac{\pi}{4}-\alpha)^{-\frac{n-1}{2}}d\alpha+(1-\frac{\pi/4}{\pi/4-\epsilon})\int_{0}^{\frac{\pi}{4}-\epsilon}(\frac{\pi}{4}-\alpha)^{-\frac{n+1}{2}}d\alpha
\]
recalling that $n=3$, that equals 
\[
-\frac{1}{\frac{\pi}{4}-\epsilon}\log\frac{\epsilon}{\frac{\pi}{4}}+\Big(1-\frac{\pi/4}{\frac{\pi}{4}-\epsilon}\Big)\frac{2}{3-1}\epsilon{}^{-\frac{3-1}{2}}=-\frac{1}{\frac{\pi}{4}-\epsilon}\log\epsilon+O(1)
\]
Thus for all $k\geq1$, $I_{-}(H_{\epsilon})$ is unbounded as $\epsilon\to0^{+}$.\\
\\
Finally, consider a general $f=af_{n,k}^{+}+bf_{n,k}^{-}$, given
by a linear combination of pure cone-symmetric and cone-antisymmetric
sectios, and assume it corresponds to a continuous valuation. Then
by the preceding argument and Proposition \ref{prop:odd N LightCone},
we must have both $a\neq0$ and $b\neq0$. When evaluated on $H_{\epsilon}$,
this would diverge as $\epsilon\to0^{+}$, since the light cone-supported
summand has a limit by Remark \ref{Remark finite one sided limits},
while the other summand diverges as was just proved. 
\end{proof}

\begin{prop}
(Even n, reduction to time-supported valuation) For even $n$, an
$(n-2)$-homogeneous valuation $\phi\in Val_{n-2}^{+}(\R^{n})^{G}$
on $\R^{n}$, if exists, has generalized Crofton measure equal to
a multiple of $f_{n,n-2}^{T}$.\end{prop}
\begin{proof}
Denote $k=n-2$, assume $\phi$ corresponds to $f=af_{n,k}^{T}+bf_{n,k}^{S}$.
\[
\phi(C_{n})=\lim_{\epsilon\to0^{+}}\phi(C_{n,\epsilon})=\lim_{\epsilon\to0^{+}}f_{-\frac{n+1}{2}}^{T}(h_{k,\epsilon}(\alpha)g_{n,n-k})
\]
Note that $h_{k,\epsilon}(\alpha)=C\eta\sin\alpha$ for $|\alpha|>\frac{\pi}{4}-\epsilon$,
and so all derivatives at $\alpha=\frac{\pi}{4}$ of $h_{k,\epsilon}(\alpha)g_{n,n-k}$
converge to a finite limit as $\epsilon\to0^{+}$and likewise $\lim_{\epsilon\to0^{+}}f_{-\frac{n+1}{2}}^{T}(h_{k,\epsilon}(\alpha)g_{n,n-k})$
is finite. We will show that $\lim_{\epsilon\to0^{+}}f_{-\frac{n+1}{2}}^{S}(h_{k,\epsilon}(\alpha)g_{n,n-k})$
is inifinite, implying $b=0$. \\
\\
Denote $H_{\epsilon}(\alpha)=h_{k,\epsilon}(\alpha)g_{n,n-k}(\alpha)$.
Write for an arbitrary function $H$ on $S^{1}$, 
\[
N(\alpha;H)=\frac{H(\alpha)-(H(\frac{\pi}{4})+\frac{1}{1!}H^{(1)}(\frac{\pi}{4})(\alpha-\frac{\pi}{4})+...+\frac{1}{m!}H^{(m)}(\frac{\pi}{4})(\alpha-\frac{\pi}{4})^{m})}{|\cos2\alpha|^{\frac{n+1}{2}}}
\]
where $m=\frac{n-2}{2}$. The integral 
\[
I(H_{\epsilon})=\int_{0}^{\frac{\pi}{4}}N(\alpha;H_{\epsilon})d\alpha
\]
equals $f_{-\frac{n+1}{2}}^{S}(h_{k,\epsilon}(\alpha)g_{n,n-k})$
up to summands corresponding to derivatives of $h_{k,\epsilon}(\alpha)g_{n,n-k}$
at the light cone, of order up to $m$. Those derivatives are uniformly
bounded as $\epsilon\to0^{+}$, since $h_{k,\epsilon}(\alpha)\to h_{k}(\alpha)$
in the $C^{m}(S^{1})$ topology by the remark following Proposition
\ref{Rem:Cm_convergence}).

We will show that $I(H_{\epsilon})$ diverges as $\epsilon\to0^{+}$.
Write
\[
I(H_{\epsilon})=\int_{0}^{\frac{\pi}{4}-\epsilon}\frac{h_{k,\epsilon}^{-}(\alpha)g_{n,n-k}(\alpha)-h_{k,\epsilon}^{+}(\alpha)g_{n,n-k}(\alpha)}{|\cos2\alpha|^{\frac{n+1}{2}}}d\alpha+\int_{0}^{\frac{\pi}{4}}N(\alpha;h_{k,\epsilon}^{+}(\alpha)g_{n,n-k}(\alpha))d\alpha
\]
Now the second integral is bounded (uniformly in $\epsilon)$, for
instance by $C|\int_{0}^{\frac{\pi}{4}}N(\alpha;h_{k}^{+}(\alpha)g_{n,n-k}(\alpha))d\alpha|$,
and the first summand is unbounded, exactly as in the case of odd
$n$ before . This concludes the proof.\end{proof}
\begin{prop}
(Non-existence of time-supported valuation with $k=n-2$) For $n$
even, $Cr(f_{n,n-2}^{T})$ is not a continuous valuation. \end{prop}
\begin{proof}
Denote $k=n-2$, $m=\frac{k}{2}=\frac{n}{2}-1$, and assume $\phi=Cr(f_{n,n-2}^{T})$
is a continuos valuation. As before $H_{\epsilon}(\alpha)=h_{k,\epsilon}g_{n,n-k}(\alpha)$.
By Remark \ref{Rem:Cm_convergence}, $h_{k,\epsilon}(\alpha)\to h_{k}(\alpha)$
as $\epsilon\to0$ in $C^{m}(S^{1})$.\\
Introduce the notations

\[
J_{j}(\alpha;H,\alpha_{0})=H(\alpha_{0})+\frac{1}{1!}H^{(1)}(\alpha_{0})(\alpha-\alpha_{0})+...+\frac{1}{j!}H^{(j)}(\alpha_{0})(\alpha-\alpha_{0})^{j}
\]
\[
N(\alpha;H,j)=\frac{H(\alpha)-J_{j}(\alpha;H,\frac{\pi}{4})}{|\cos2\alpha|^{j+\frac{3}{2}}}
\]
and 
\[
I(u)=\int_{\frac{\pi}{4}}^{\frac{\pi}{2}}N(\alpha;u,m)d\alpha
\]
Observe that $H_{\epsilon}\to H$ in $C^{m}(S^{1})$ as well, so all
the derivatives satisfy $H_{\epsilon}^{(j)}(\frac{\pi}{4})\to H^{(j)}(\frac{\pi}{4})$
for $j\leq m$ as $\epsilon\to0$. We will show that 
\[
\lim_{\epsilon\to0^{+}}f_{-\frac{n+1}{2}}^{T}(H_{\epsilon})\neq\lim_{\epsilon\to0^{-}}f_{-\frac{n+1}{2}}^{T}(H_{\epsilon})
\]
Equivalently, due to $C^{m}$ convergence, we will show that $I(H_{\epsilon})$
has different one-sided limits. \\
Denote 
\[
u_{\epsilon}(\alpha)=\frac{h_{k,\epsilon}(\alpha)}{\sin\alpha}
\]
Recall that 
\[
u_{\epsilon}(\alpha)=\Bigg\{\begin{array}{c}
A_{k}\eta,\alpha\geq\frac{\pi}{4}-\epsilon\\
A_{k}\eta-2\eta\int_{\eta\tan\alpha}^{1}(1-t^{2})^{\frac{k-3}{2}}dt+\frac{2}{k-1}\cot\alpha(1-\eta^{2}\tan^{2}\alpha)^{\frac{k-1}{2}},0\leq\alpha\leq\frac{\pi}{4}-\epsilon
\end{array}
\]
where $A_{k}=\int_{-\pi/2}^{\pi/2}\cos^{k-2}\phi d\phi$, implying
\[
\lim_{\epsilon\to0^{+}}I(u_{\epsilon})=\lim_{\epsilon\to0^{+}}I(A_{k}\eta)=0
\]
Now write $H_{\epsilon}=t(\alpha)u_{\epsilon}(\alpha)$ where $t(\alpha)=g_{n,n-k}(\alpha)\sin\alpha$.
According to Lemma \ref{lem:jets_technical}, we may write 
\[
H(\alpha)-J_{j}(\alpha;H,\frac{\pi}{4})=t(\frac{\pi}{4})(u_{\epsilon}(x)-J_{m}(\alpha;u_{\epsilon},\frac{\pi}{4}))+u_{\epsilon}(\alpha)R_{m+1}(\alpha)+O(C_{\epsilon}|\alpha-\frac{\pi}{4}|^{m+1})
\]
where $R_{m+1}(\alpha)=t(\alpha)-J_{m}(\alpha;t,\frac{\pi}{4})$,
and the constant $C_{\epsilon}$ in the error term is bounded by 
\[
C_{m}\sup_{1\leq j\leq m}|u_{\epsilon}^{(j)}|
\]
with 
\[
C_{m}=m\sup_{0\leq j\leq m+1}|(g_{n,n-k}(\alpha)\sin\alpha)^{(j)}|
\]
where everywhere $\alpha\in[0,\frac{\pi}{2}]$. By the convergence
of $u_{\epsilon}(\alpha)\to A_{k}$ in $C^{m}[\frac{\pi}{4},\frac{\pi}{2}]$,
we conclude that $C_{\epsilon}\to0$ as $\epsilon\to0$.\\
Since $|R_{m+1}(\alpha)|\leq C|\alpha-\frac{\pi}{4}|^{m+1}$, and
$u_{\epsilon}$ converges in $C[\frac{\pi}{4},\frac{\pi}{2}]$, the
integral 
\[
\int_{\frac{\pi}{4}}^{\frac{\pi}{2}}\frac{u_{\epsilon}(\alpha)R_{m+1}(\alpha)}{|\cos2\alpha|^{m+\frac{3}{2}}}d\alpha
\]
has a limit as $\epsilon\to0$. Also, the integral 
\[
\int_{\frac{\pi}{4}}^{\frac{\pi}{2}}\frac{O(C_{\epsilon}|\alpha-\frac{\pi}{4}|^{m+1})}{|\cos2\alpha|^{m+\frac{3}{2}}}d\alpha
\]
converges to $0$ as $\epsilon\to0$. We conclude that $I(H_{\epsilon})-t(\frac{\pi}{4})I(u_{\epsilon})$
converges, and thus it suffices to show that the functional $I(u_{\epsilon})$
has different one-sided limits. We will verify that
\[
\lim_{\epsilon\to0^{-}}I(u_{\epsilon})\neq0
\]
From now on $\epsilon<0$. We will use the approximations 
\[
\eta=\tan(\frac{\pi}{4}+\epsilon)=1+2\epsilon+O(\epsilon^{2})
\]
\[
1-\eta^{4}=-8\epsilon+O(\epsilon^{2})
\]
\[
(1-\eta^{2})^{\frac{1}{2}}=(-4\epsilon+O(\epsilon^{2}))^{\frac{1}{2}}=2|\epsilon|^{\frac{1}{2}}+O(|\epsilon|)
\]
Then for $\alpha<\frac{\pi}{4}-\epsilon$,
\[
u_{\epsilon}'(\alpha)=-\frac{2}{k-1}(1-\eta^{2}\tan^{2}\alpha)^{m-\frac{1}{2}}\frac{1}{\sin^{2}\alpha}
\]
It follows by induction that for $\alpha\in(\frac{\pi}{4}+\epsilon,\frac{\pi}{4}-\epsilon)$
and $j\geq1$, 
\[
u_{\epsilon}^{(j)}(\alpha)=(-1)^{j}\frac{2}{k-1}\frac{1}{\sin^{2}\alpha}\frac{1}{2^{j-1}}\frac{(2m-1)!!}{(2m-2j+1)!!}\eta^{2j-2}(1-\eta^{2}\tan^{2}\alpha)^{m+\frac{1}{2}-j}\Bigg(\frac{2\tan\alpha}{\cos^{2}\alpha}\Bigg)^{j-1}+
\]
\[
+O\Big((1-\eta^{2}\tan^{2}\alpha)^{m+\frac{3}{2}-j}\Big)=
\]
\[
=(-1)^{j}\frac{2}{k-1}\frac{(2m-1)!!}{(2m-2j+1)!!}\eta^{2j-2}\frac{\sin^{j-3}\alpha}{\cos^{3j-3}\alpha}(1-\eta^{2}\tan^{2}\alpha)^{m+\frac{1}{2}-j}+O\Big((1-\eta^{2}\tan^{2}\alpha)^{m+\frac{3}{2}-j}\Big)
\]
in particular, for $1\leq j\leq m$ and $\alpha\in(\frac{\pi}{4}+\epsilon,\frac{\pi}{4}-\epsilon)$,
$|1-\eta^{2}\tan^{2}\alpha|=O(|\epsilon|)$ so 
\begin{equation}
|u_{\epsilon}^{(j)}(\alpha)|=O(|\eps|^{m+\frac{1}{2}-j})\label{eq:bound on derivative}
\end{equation}
It therefore also holds that 
\begin{equation}
|u_{\epsilon}(\frac{\pi}{4})-A_{k}\eta|=|u_{\epsilon}(\frac{\pi}{4})-u_{\eps}(\frac{\pi}{4}-\epsilon)|=O(|\eps|^{m+\frac{1}{2}})\label{eq:bound on 0 derivative}
\end{equation}
Write 
\begin{equation}
I(u_{\epsilon})=\int_{\frac{\pi}{4}-\epsilon}^{\frac{\pi}{2}}\frac{A_{k}\eta-J_{m}(\alpha;u_{\epsilon},\frac{\pi}{4})}{(\alpha-\frac{\pi}{4})^{m+\frac{3}{2}}}w(\alpha)d\alpha+\int_{\frac{\pi}{4}}^{\frac{\pi}{4}-\epsilon}\frac{u_{\epsilon}(\alpha)-J_{m}(\alpha;u_{\epsilon},\frac{\pi}{4})}{(\alpha-\frac{\pi}{4}){}^{m+\frac{3}{2}}}w(\alpha)d\alpha\label{eq:I_epsilon}
\end{equation}
where 
\[
w(\alpha)=\frac{|\alpha-\frac{\pi}{4}|^{m+\frac{3}{2}}}{|\cos2\alpha|^{m+\frac{3}{2}}}
\]
is a $C^{\infty}$, strictly positive function in $[0,\frac{\pi}{2}]$.
Now integrate by parts: we integrate the denominator and differentiate
the numerator.
\[
\int_{\frac{\pi}{4}-\epsilon}^{\frac{\pi}{2}}\frac{A_{k}\eta-J_{m}(\alpha;u_{\epsilon},\frac{\pi}{4})}{(\alpha-\frac{\pi}{4})^{m+\frac{3}{2}}}w(\alpha)d\alpha=-\frac{1}{m+\frac{1}{2}}\frac{A_{k}\eta-J_{m}(\frac{\pi}{2};u_{\epsilon},\frac{\pi}{4})}{(\frac{\pi}{4})^{m+\frac{1}{2}}}w(\frac{\pi}{2})
\]
\[
+\frac{1}{m+\frac{1}{2}}\frac{A_{k}\eta-J_{m}(\frac{\pi}{4}-\epsilon;u_{\epsilon},\frac{\pi}{4})}{|\epsilon|^{m+\frac{1}{2}}}w(\frac{\pi}{4}-\epsilon)+\frac{1}{m+\frac{1}{2}}\int_{\frac{\pi}{4}-\epsilon}^{\frac{\pi}{2}}\frac{J_{m-1}(\alpha;-u_{\epsilon}',\frac{\pi}{4})}{(\alpha-\frac{\pi}{4})^{m+\frac{1}{2}}}w(\alpha)d\alpha+
\]
\[
+\frac{1}{m+\frac{1}{2}}\int_{\frac{\pi}{4}-\epsilon}^{\frac{\pi}{2}}\frac{A_{k}\eta-J_{m}(\alpha;u_{\epsilon},\frac{\pi}{4})}{(\alpha-\frac{\pi}{4})^{m+\frac{1}{2}}}w'(\alpha)d\alpha
\]
the first summand is $o(1)$ as $\epsilon\to0^{-}$, since $u_{\epsilon}\to A_{k}\in C^{m}[\frac{\pi}{4},\frac{\pi}{2}]$,
so $J_{m}(\frac{\pi}{2};u_{\epsilon},\frac{\pi}{4})\to J_{m}(\frac{\pi}{2};A_{k},\frac{\pi}{4})=A_{k}$
as $\epsilon\to0^{-}$. Let us verify that the last summand is also
$o(1)$. Indeed 
\[
\Bigg|\int_{\frac{\pi}{4}-\epsilon}^{\frac{\pi}{2}}\frac{A_{k}\eta-J_{m}(\alpha;u_{\epsilon},\frac{\pi}{4})}{(\alpha-\frac{\pi}{4})^{m+\frac{1}{2}}}w'(\alpha)d\alpha\Bigg|\leq C\int_{\frac{\pi}{4}-\epsilon}^{\frac{\pi}{2}}\Bigg(\frac{|A_{k}\eta-u_{\epsilon}(\frac{\pi}{4})|}{(\alpha-\frac{\pi}{4})^{m+\frac{1}{2}}}+\frac{1}{j!}\sum_{j=1}^{m}\frac{|u_{\epsilon}^{(j)}(\frac{\pi}{4})|}{(\alpha-\frac{\pi}{4})^{m+\frac{1}{2}-j}}\Bigg)d\alpha
\]
This can be integrated explicitly. The terms corresponding to $\frac{\pi}{2}$
are all $o(1)$ again since $u_{\epsilon}\to A_{k}\in C^{m}[\frac{\pi}{4},\frac{\pi}{2}]$,
while the terms corresponding to $\frac{\pi}{4}-\epsilon$ are all
$O(|\epsilon|)$ by estimates \ref{eq:bound on derivative} and \ref{eq:bound on 0 derivative}.
Therefore,
\[
\int_{\frac{\pi}{4}-\epsilon}^{\frac{\pi}{2}}\frac{A_{k}\eta-J_{m}(\alpha;u_{\epsilon},\frac{\pi}{4})}{(\alpha-\frac{\pi}{4})^{m+\frac{3}{2}}}w(\alpha)d\alpha=
\]
\[
=\frac{1}{m+\frac{1}{2}}\Bigg(\frac{A_{k}\eta-J_{m}(\frac{\pi}{4}-\epsilon;u_{\epsilon},\frac{\pi}{4})}{|\epsilon|^{m+\frac{1}{2}}}w(\frac{\pi}{4}-\epsilon)+\int_{\frac{\pi}{4}-\epsilon}^{\frac{\pi}{2}}\frac{J_{m-1}(\alpha;-u_{\epsilon}',\frac{\pi}{4})}{(\alpha-\frac{\pi}{4})^{m+\frac{1}{2}}}w(\alpha)d\alpha\Bigg)+o(1)
\]
Similarly, we may integrate by parts the second summand of $I(u_{\epsilon})$
in equation \ref{eq:I_epsilon} as follows: 
\[
\int_{\frac{\pi}{4}}^{\frac{\pi}{4}-\epsilon}\frac{u_{\epsilon}(\alpha)-J_{m}(\alpha;u_{\epsilon},\frac{\pi}{4})}{(\alpha-\frac{\pi}{4}){}^{m+\frac{3}{2}}}w(\alpha)d\alpha=-\frac{1}{m+\frac{1}{2}}\frac{A_{k}\eta-J_{m}(\frac{\pi}{4}-\epsilon;u_{\epsilon},\frac{\pi}{4})}{|\epsilon|^{m+\frac{1}{2}}}w(\frac{\pi}{4}-\epsilon)+
\]
\[
+\frac{1}{m+\frac{1}{2}}\int_{\frac{\pi}{4}}^{\frac{\pi}{4}-\epsilon}\frac{u_{\epsilon}'(\alpha)+J_{m-1}(\alpha;-u_{\epsilon}',\frac{\pi}{4})}{(\alpha-\frac{\pi}{4})^{m+\frac{1}{2}}}w(\alpha)d\alpha
\]
the $\frac{\pi}{4}$-boundary term vanishes since $u_{\epsilon}$
is $C^{\infty}$ near $\frac{\pi}{4}$. Thus
\[
I(u_{\epsilon})=\frac{1}{m+\frac{1}{2}}\Bigg(\int_{\frac{\pi}{4}-\epsilon}^{\frac{\pi}{2}}\frac{J_{m-1}(\alpha;-u_{\epsilon}',\frac{\pi}{4})}{(\alpha-\frac{\pi}{4})^{m+\frac{1}{2}}}w(\alpha)d\alpha+\int_{\frac{\pi}{4}}^{\frac{\pi}{4}-\epsilon}\frac{u_{\epsilon}'(\alpha)-J_{m-1}(\alpha;u_{\epsilon}',\frac{\pi}{4})}{(\alpha-\frac{\pi}{4})^{m+\frac{1}{2}}}w(\alpha)d\alpha\Bigg)+o(1)
\]
so we should show that the expression in the brackets does not vanish
as $\epsilon\to0^{-}$. Repeatedly applying integration by parts as
we did for equation \ref{eq:I_epsilon}, we end up having to show
that
\[
J(\epsilon)=\int_{\frac{\pi}{4}-\epsilon}^{\frac{\pi}{2}}\frac{-u_{\epsilon}^{(m)}(\frac{\pi}{4})}{(\alpha-\frac{\pi}{4})^{\frac{3}{2}}}w(\alpha)d\alpha+\int_{\frac{\pi}{4}}^{\frac{\pi}{4}-\epsilon}\frac{u_{\epsilon}^{(m)}(\alpha)-u_{\epsilon}^{(m)}(\frac{\pi}{4})}{(\alpha-\frac{\pi}{4})^{\frac{3}{2}}}w(\alpha)d\alpha
\]
does not converge to $0$ as $\epsilon\to0^{-}$. \\
\\
Recall that
\[
u_{\epsilon}^{(m)}(\alpha)=(-1)^{m}\frac{2}{k-1}(2m-1)!!\eta{}^{2m-2}(1-\eta^{2}\tan^{2}\alpha)^{\frac{1}{2}}\frac{\sin^{m-3}\alpha}{\cos^{3m-3}\alpha}+O\Big((1-\eta^{2}\tan^{2}\alpha)^{3/2}\Big)
\]
in particular, 
\[
u_{\epsilon}^{(m)}(\frac{\pi}{4})=(-1)^{m}\frac{2}{k-1}(2m-1)!!(1-\eta^{2})^{\frac{1}{2}}\eta^{2m-2}2^{m}+O(|\epsilon|^{3/2})=
\]
\[
=(-1)^{m}\frac{2}{k-1}(2m-1)!!\eta^{2m-2}2^{m+1}|\epsilon|^{1/2}+O(|\epsilon|^{3/2})
\]
We will also need the finer estimate
\[
u_{\epsilon}^{(m)}(\alpha)-u_{\epsilon}^{(m)}(\frac{\pi}{4})=(-1)^{m}\frac{2}{k-1}(2m-1)!!\eta^{2m-2}\Big((1-\eta^{2}\tan^{2}\alpha)^{\frac{1}{2}}\frac{\sin^{m-3}\alpha}{\cos^{3m-3}\alpha}-(1-\eta^{2})^{\frac{1}{2}}2{}^{m}\Big)+
\]
\[
+O\Big(\alpha-\frac{\pi}{4}\Big)
\]
which is obtained by writing 
\[
u_{\epsilon}^{(m)}(\alpha)=(-1)^{m}\frac{2}{k-1}(2m-1)!!\eta{}^{2m-2}(1-\eta^{2}\tan^{2}\alpha)^{\frac{1}{2}}\frac{\sin^{m-3}\alpha}{\cos^{3m-3}\alpha}+s_{\epsilon}(\alpha)(1-\eta^{2}\tan^{2}\alpha)^{3/2}
\]
where $s_{\epsilon}(\alpha)\in C^{1}(\frac{\pi}{5},\frac{\pi}{3})$
is uniformly bounded in $C^{1}(\frac{\pi}{5},\frac{\pi}{3})$. Then
the error term in $u_{\epsilon}^{(m)}(\alpha)-u_{\epsilon}^{(m)}(\frac{\pi}{4})$
is easily seen to equal
\[
+O\Big((1-\eta^{2})^{3/2}-(1-\eta^{2}\tan^{2}\alpha)^{3/2}\Big)+O\Big(\alpha-\frac{\pi}{4}\Big)
\]
and since $(1-\eta^{2}\tan^{2}\alpha)^{3/2}$ is $C^{1}(\frac{\pi}{5},\frac{\pi}{3})$
and uniformly bounded, one has 
\[
(1-\eta^{2})^{3/2}-(1-\eta^{2}\tan^{2}\alpha)^{3/2}=O\Big(\alpha-\frac{\pi}{4}\Big)
\]
Integrating the first summand of $J(\epsilon)$ by parts, we get that
\[
\int_{\frac{\pi}{4}-\epsilon}^{\frac{\pi}{2}}\frac{-u_{\epsilon}^{(m)}(\frac{\pi}{4})}{(\alpha-\frac{\pi}{4})^{\frac{3}{2}}}w(\alpha)d\alpha=-u_{\epsilon}^{(m)}(\frac{\pi}{4})\frac{2}{|\epsilon|^{\frac{1}{2}}}w(\frac{\pi}{4})+o(1)
\]
\[
=(-1)^{m+1}\frac{2}{k-1}(2m-1)!!\eta^{2m-2}w(\frac{\pi}{4})\Big(2^{m+2}+o(1)\Big)
\]
 while 
\[
\int_{\frac{\pi}{4}}^{\frac{\pi}{4}-\epsilon}\frac{u_{\epsilon}^{(m)}(\alpha)-u_{\epsilon}^{(m)}(\frac{\pi}{4})}{(\alpha-\frac{\pi}{4})^{\frac{3}{2}}}w(\alpha)d\alpha=
\]
\[
=\int_{\frac{\pi}{4}}^{\frac{\pi}{4}-\epsilon}\frac{(-1)^{m}\frac{2}{k-1}(2m-1)!!\eta^{2m-2}\Big((1-\eta^{2}\tan^{2}\alpha)^{\frac{1}{2}}\frac{\sin^{m-3}\alpha}{\cos^{3m-3}\alpha}-(1-\eta^{2})^{\frac{1}{2}}2{}^{m}\Big)}{(\alpha-\frac{\pi}{4})^{\frac{3}{2}}}w(\alpha)d\alpha+o(1)=
\]
\[
=(-1)^{m}\frac{2}{k-1}(2m-1)!!\eta^{2m-2}w(\frac{\pi}{4})\int_{\frac{\pi}{4}}^{\frac{\pi}{4}-\epsilon}\frac{(1-\eta^{2}\tan^{2}\alpha)^{\frac{1}{2}}\frac{\sin^{m-3}\alpha}{\cos^{3m-3}\alpha}-(1-\eta^{2})^{\frac{1}{2}}2{}^{m}}{(\alpha-\frac{\pi}{4})^{\frac{3}{2}}}d\alpha+o(1)
\]
So it remains to show that 
\[
-2^{m+2}+\int_{\frac{\pi}{4}}^{\frac{\pi}{4}-\epsilon}\frac{(1-\eta^{2}\tan^{2}\alpha)^{\frac{1}{2}}\frac{\sin^{m-3}\alpha}{\cos^{3m-3}\alpha}-(1-\eta^{2})^{\frac{1}{2}}2{}^{m}}{(\alpha-\frac{\pi}{4})^{\frac{3}{2}}}d\alpha\nrightarrow0
\]
Since 
\[
\frac{\sin^{m-3}\alpha}{\cos^{3m-3}\alpha}=2^{m}+O(\alpha-\frac{\pi}{4})
\]
this boils down to
\[
-4+\int_{\frac{\pi}{4}}^{\frac{\pi}{4}-\epsilon}\frac{(1-\eta^{2}\tan^{2}\alpha)^{\frac{1}{2}}-(1-\eta^{2})^{\frac{1}{2}}}{(\alpha-\frac{\pi}{4})^{\frac{3}{2}}}d\alpha\nrightarrow0
\]
The integral is non-positive. This concludes the proof. 
\end{proof}

\section{Applications}

Recently in \cite{Parapatis-Wennerer}, some negative results on continuity
properties of classical constructions in the theory of valuations
were proved. We will now explain how some of those results can be
seen immediately from the classification of Lorentz-invariant valuations.

\subsection{The image of the Klain imbedding is not closed}

Denote by $\phi_{n,k}^{\pm}\in Val_{k}^{ev,-\infty}(V)^{G}$ the two
independent generalized valuations that we found. The generalized
Klain sections $Kl(\phi_{n,k}^{\pm})\in\Gamma(K^{n,k})$ for $1\leq k\leq n-2$
are in fact continuous sections of the Klain bundle, that do not correspond
to a continuous valuation. They do belong to the closure (in the $C^{0}$
topology) of the image of the Klain imbedding on continuous valuations.

\subsection{The Fourier transform does not extend to continuous valuations}

The Fourier transform on smooth even valuations extends to the space
of generalized smooth valuations by self-adjointness (see \cite{bernig-fu-convolution}):
For $\phi\in Val_{k}^{ev,-\infty}(V)$, we define $\F\phi\in Val_{n-k}^{ev,-\infty}(V^{*})\otimes D(V)$
by letting for all $\psi\in Val_{k}^{ev,\infty}(V^{*})$ 
\[
\langle\F\phi,\psi\rangle=\langle\phi,\F\psi\rangle
\]
It is a $GL(V)$-equivariant involution (in the sense that $(\F_{V^{*}}\otimes Id)\circ\F_{V}=Id$).
Restricting to $G=SO^{+}(n-1,1)$, we get a $G$-equivariant involution
\[
\F:Val_{k}^{ev,-\infty}(V)\to Val_{n-k}^{ev,-\infty}(V)
\]
which restricts to the usual ($G$-equivariant) Fourier transform
on smooth even valuations. \\
Let $\phi_{n,n-1}^{\pm}\in Val_{n-1}^{ev}(V)^{G}$ be the cone-symmetric
and cone-antisymmetric continuous valuations that we found. It follows
by equivariance that 
\[
\F(\phi_{n,n-1}^{\pm})\in Val_{1}^{ev,-\infty}(V)^{G}
\]
Since $Val_{1}^{ev,-\infty}(V)^{G}$ contains no non-trivial continuous
valuations when $n\geq3$ , it follows that the Fourier transform
does not extend by continuity to continuous valuations for $n\geq3$.

\appendix

\section{A technical lemma}

We denote by $J^{m}(x;f,a)$ the Taylor polynomial of order $m$ for
the function $f$ around $a$.
\begin{lem}
\label{lem:jets_technical}For $w\in C^{\infty}(\R)$ and $h\in C^{m}(\R)$
it holds in any fixed compact interval $I$ around $0$ that 
\[
w(x)h(x)-J_{m}(x;wh,0)=w(0)(h(x)-J_{m}(x;h,0))+h(x)R_{m+1}(x)+O(|x|^{m+1})
\]
as $x\to0$, where $R_{m+1}(x)=w(x)-J_{m}(x;w,0)$. More precisely,
if $|h^{(j)}(x)|\leq H_{j}$ for all $x\in I$ and $0\leq j\leq m$
and $|w^{(j)}(x)|\leq W$ for all $x\in I$ and $j\leq m+1$ then
$O(|x|^{m+1})\leq C_{m,I}(H_{m}+...+H_{1})W|x|^{m+1}$.\end{lem}
\begin{proof}
Write $J_{m}(f)$ for $J_{m}(x;f,0)$. Then 
\[
h=J_{m}(h)+e_{1}(x)
\]
\[
w-w(0)=J_{m}(w-w(0))+e_{2}(x)
\]
where 
\[
|e_{1}(x)|\leq c_{m,I}H_{m}|x|^{m}
\]
\[
|e_{2}(x)|\leq c'_{m,I}W|x|^{m+1}
\]
so
\[
wh=(w-w(0))h+w(0)h=J_{m}(w-w(0))h+w(0)h+hR_{m+1}=
\]
\[
=J_{m}(w-w(0))(J_{m}(h)+O(H_{m}|x|^{m}))+w(0)h+hR_{m+1}=
\]
\[
=J_{m}(w-w(0))J_{m}(h)+w(0)h+O\Big(H_{m}W|x|^{m+1}\Big)+hR_{m+1}
\]
the last equality since $J_{m}(w-w(0))=O(W_{1}|x|)$. Note that 
\[
J_{m}(w-w(0))J_{m}(h)=J_{m}((w-w(0))h)+O\Big((H_{m}+...+H_{1})W|x|^{m+1}\Big)=
\]
\[
=J_{m}(wh)-w(0)J_{m}(h)+O\Big((H_{m}+...+H_{1})W|x|^{m+1}\Big)
\]
so
\[
wh=J_{m}(wh)-w(0)J_{m}(h)+w(0)h+O((H_{m}+...+H_{1})W|x|^{m+1})+hR_{m+1}
\]
as claimed.\end{proof}


\begin{thebibliography}{10}
\bibitem{alesker-adv-2000} Alesker, Semyon; \emph{On P. McMullen's
conjecture on translation invariant valuations,} Adv. Math. 155 (2000),
no. 2, 239-263.

\bibitem{alesker-gafa-04} Alesker, Semyon; \emph{The multiplicative
structure on polynomial continuous valuations,} Geom. Funct. Anal.
14(1) (2004), 1-26.

\bibitem{Alesker-jdg}Alesker, Semyon; \emph{Hard Lefschetz theorem
for valuations, complex integral geometry, and unitarily invariant
valuations}, Journal of Differential Geometry 63 (2003), 63-95.

\bibitem{AleskerBernstein} Alesker, Semyon; Bernstein, Joseph; \emph{Range
characterization of the cosine transform on higher Grassmannians},
Advances in Math., 184 (2004), 367-379.

\bibitem{bernig-sun}Bernig, Andreas;\emph{ A Hadwiger-type theorem
for the special unitary group,} Geom. Funct. Anal. 19 (2009), no.
2, 356-372.

\bibitem{bernig-exception} Bernig, Andreas; \emph{Integral geometry
under $G_{2}$ and $Spin(7)$}, Israel J. Math. 184 (2011), 301-316.

\bibitem{bernig-quaternion} Bernig, Andreas; \emph{Invariant valuations
on quaternionic vector spaces,} Journal de l'Institut de Mathématiques
de Jussieu 11 (2012), 467-499.

\bibitem{bernig-fu-convolution}Bernig, Andreas; Fu, Joseph H. G.
\emph{Convolution of convex valuations,} Geom. Dedicata (2006) 123:153-169.

\bibitem{bernig-fu-annals} Bernig, Andreas; Fu, Joseph H. G.\emph{
Hermitian integral geometry,} Ann. of Math. (2) 173 (2011), no. 2,
907-945.

\bibitem{fu-jdg-algebra} Fu, Joseph H. G.; \emph{Structure of the
unitary valuation algebra}, J. Differential Geom. 72 (2006), no. 3,
509-533.

\bibitem{Fuks}D.B. Fuks, Cohomology of Infinite-Dimensional Lie Algebras,
Contemporary Soviet Mathematics, Steklov Institute, Moscow, 1986.

\bibitem{Cosine Transform}Goodey, Paul R.; Weil, Wolfgang; \emph{Centrally
symmetric convex bodies and the spherical Radon transform}, J. Differential
Geom. 35 (1992) 675-688.

\bibitem{Radon}Guillemin, Victor;\emph{ On some results of Gelfand
in integral geometry}, Proc.Symp.Pure Math. 43(1985),149-155.

\bibitem{hadwiger-book} H. Hadwiger, Vorlesungen ?ber Inhalt, Oberfläche
und Isoperimetrie. (German) Springer-Verlag, Berlin-Göttingen-Heidelberg
1957.

\bibitem{klain} Klain, Daniel A.; \emph{A short proof of Hadwiger's
characterization theorem,} Mathematika 42 (1995), no. 2, 329-339.

\bibitem{ludwig-adv-math-2002} Ludwig, Monika; \emph{Projection bodies
and valuations}, Adv. Math. 172 (2002), no. 2, 158-168.

\bibitem{ludwig-amer-j-math-2006} Ludwig, Monika; \emph{Intersection
bodies and valuations,} Amer. J. Math. 128 (2006), no. 6, 1409-1428.

\bibitem{ludwig-reitzner-99} Ludwig, Monika; Reitzner, Matthias;
\emph{A characterization of affine surface area}, Adv. Math. 147 (1999),
no. 1, 138-172.

\bibitem{ludwig-reitzner-annals2010} Ludwig, Monika; Reitzner, Matthias;
\emph{A classification of SL(n) invariant valuations,} Ann. of Math.
(2) 172 (2010), no. 2, 1219-1267.

\bibitem{mcmullen-decomp}McMullen, Peter; \emph{Valuations and Euler-type
relations on certain classes of convex polytopes}, Proc. London Math.
Soc. (3) 35:1 (1977), 113-135.

\bibitem{Parapatis-Wennerer} Parapatits, Lukas; Wannerer, Thomas;\emph{
On the inverse Klain map, }arXiv:1206.5370 {[}math.MG{]}.

\bibitem{schneider-book} Schneider, Rolf; Convex bodies: the Brunn-Minkowski
theory. Encyclopedia of Mathematics and its Applications, 44. Cambridge
University Press, Cambridge, 1993.

\bibitem{schneider-simple} Schneider, Rolf; \emph{Simple valuations
on convex bodies}, Mathematika 43 (1996), no. 1, 32-39.

\bibitem{schneider-schuster-imnr-2006} Schneider, Rolf; Schuster,
Franz E.;\emph{ Rotation equivariant Minkowski valuations,} Int. Math.
Res. Not. 2006, Art. ID 72894, 20 pp.

\bibitem{schuster-wannerer} Schuster, Franz E.; Wannerer, Thomas;\emph{
$GL(n)$ contravariant Minkowski valuations}, Trans. Amer. Math. Soc.
364 (2012), no. 2, 815-826.

\bibitem{SM}Sherman, Jack; Morrison, Winifred J.; \emph{Adjustment
of an Inverse Matrix Corresponding to a Change in One Element of a
Given Matrix, }Annals of Mathematical Statistics 21 (1), 1950, 124\textendash{}127.\end{thebibliography}
\end{document}